\documentclass[11pt]{amsart}
\usepackage{amsmath,amssymb,,latexsym,esint,cite}
\usepackage{verbatim}
\usepackage[left=3.1cm,right=3.1cm,top=3.2cm,bottom=3.2cm]{geometry}

\usepackage{color,enumitem,graphicx}
\usepackage[colorlinks=true,urlcolor=blue, citecolor=red,linkcolor=blue,linktocpage,pdfpagelabels, bookmarksnumbered,bookmarksopen]{hyperref}

\usepackage[hyperpageref]{backref}
\usepackage[english]{babel}

\newcommand{\M}{{\mathcal M}}
\newcommand{\N}{\mathbb N}

\newcommand{\pnorm}[2][]{\if #1'' \left|#2\right|_p \else \left|#2\right|_{#1} \fi}
\newcommand{\R}{\mathbb R}

\newenvironment{enumroman}{\begin{enumerate}
\renewcommand{\theenumi}{$(\roman{enumi})$}
\renewcommand{\labelenumi}{$(\roman{enumi})$}}{\end{enumerate}}

\newtheorem{lemma}{Lemma}[section]
\newtheorem{proposition}[lemma]{Proposition}
\newtheorem{theorem}[lemma]{Theorem}

\theoremstyle{definition}

\theoremstyle{Remark}
\newtheorem{Remark}[lemma]{Remark}

\numberwithin{equation}{section}

\title[non-local elliptic systems with Hardy-Littlewood-Sobolev critical nonlinearities]{Existence results
 for non-local elliptic systems with Hardy-Littlewood-Sobolev critical nonlinearities }

\author[Q.Y.\ Hong, Y. Yang, X.D.\ Shang]{QianYu\ Hong,  Yang Yang $^{\dag}$, Xudong Shang}

\address[Q.Y.\ Hong]{School of Science
    \newline\indent
    Jiangnan University
    \newline\indent
    Wuxi, Jiangsu 214122, China}
\email{1031369190@qq.com}
\address[Y. Yang]{School of Science
    \newline\indent
    Jiangnan University
    \newline\indent
    Wuxi, Jiangsu 214122, China}
\email{yynjnu@126.com}
\address[X.D.\ Shang]{ School of Mathematics,
    \newline\indent
     Nanjing
Normal University Taizhou College
    \newline\indent
    Taizgou, Jiangsu 225300, China}
\email{xudong-shang@163.com}

\thanks{$^{\dag}$ Corresponding author. Project supported by NSFC(Nos. 11501252, 11571176, 11601234), Natural Science Foundation of Jiangsu
Province of China for Young Scholar (No.BK20160571), and Qinglan
Project of Jiangsu Province (2016, 2018)}

\begin{document}

\begin{abstract}
In this article, we study  the following nonlinear doubly nonlocal problem involving the fractional Laplacian
in the sense of Hardy-Littlewood-Sobolev inequality
\begin{equation*}
\left\{\begin{aligned}
(-\Delta)^s u & =
au+bv+\frac{2p}{p+q}\int_{\Omega}\frac{|v(y)|^q}{|x-y|^\mu}dy|u|^{p-2}u+2\xi_1\int_{\Omega}\frac{|u(y)|^{2^*_\mu}}{|x-y|^\mu}dy|u|^{2^*_\mu-2}u,&&
\text{in } \Omega;\\
(-\Delta)^s v & =
bu+cv+\frac{2q}{p+q}\int_{\Omega}\frac{|u(y)|^p}{|x-y|^\mu}dy|v|^{q-2}v+2\xi_2\int_{\Omega}\frac{|v(y)|^{2^*_\mu}}{|x-y|^\mu}dy|v|^{2^*_\mu-2}v,&&
\text{in } \Omega;\\
u &=v=0,\text{ in }   \R^N\setminus\Omega,
\end{aligned}\right.
\end{equation*}
where $\Omega$ is a smooth bounded domain in $\R^N$, $N>2s$, $s\in(0,1)$, $\xi_1,\xi_2\geq 0$, $(-\Delta)^s$ is
the well known fractional Laplacian, $\mu\in(0,N)$, $1<p,q\leq 2^*_\mu$ where $2^*_\mu=\frac{2N-\mu}{N-2s}$
is the upper critical exponent in the Hardy-Littlewood-Sobolev inequality. Under suitable assumptions on
different  parameters  $p, q, \xi_1,$ and $ \xi_2$, we are able to prove  some existence   and multiplicity
results  for the above equation by variational methods.
\end{abstract}

\maketitle
\begin{center}
    \begin{minipage}{12cm}
        \small
        \tableofcontents
    \end{minipage}
\end{center}
\medskip

\section{Introduction and main results}
Let $\Omega\subset \R^N $ be a bounded domain with smooth boundary $\partial \Omega$ (at least\ $C^2$), $N>2s$
and $s\in (0,1)$. We consider the following nonlinear doubly nonlocal system involving the
fractional Laplacian:
\begin{equation} \label{1.1}
\left\{\begin{aligned}
(-\Delta)^s u & =
au+bv+\frac{2p}{p+q}\int_{\Omega}\frac{|v(y)|^q}{|x-y|^\mu}dy|u|^{p-2}u+2\xi_1\int_{\Omega}\frac{|u(y)|^{2^*_\mu}}{|x-y|^\mu}dy|u|^{2^*_\mu-2}u,&&
\text{in } \Omega;\\
(-\Delta)^s v & =
bu+cv+\frac{2q}{p+q}\int_{\Omega}\frac{|u(y)|^p}{|x-y|^\mu}dy|v|^{q-2}v+2\xi_2\int_{\Omega}\frac{|v(y)|^{2^*_\mu}}{|x-y|^\mu}dy|v|^{2^*_\mu-2}v,&&
\text{in } \Omega;\\
u &=v=0,\text{ in }   \R^N\setminus\Omega,
\end{aligned}\right.
\end{equation}
 $\mu\in(0,N)$ , $\xi_1,\xi_2\geq 0$ and $1<p,q\leq 2^*_\mu$ where $2^*_\mu=\frac{2N-\mu}{N-2s}$  is the upper critical exponent in the
Hardy-Littlewood-Sobolev inequality. $(-\Delta)^s$ is the fractional Laplacian operator defined as
\begin{equation*}
(-\Delta)^s u(x)=-P.V.\int_{\R^N}\frac{u(x)-u(y)}{|x-y|^{N+2s}}dy
\end{equation*}
where $P.V.$ denotes the Cauchy principal value. This type of operators arise in  many different contexts, such as, among the others,  physical phenomena, stochastic processes, fluid dynamics, dynamical systems, elasticity, obstacle problems, conservation laws, ultra-relativistic limits of quantum mechanics, quasi-geostrophic flows, multiple scattering, materials science and water waves. For more details, we refer to \cite{MRApplebaum,MRGarroni1}.
For any measurable function $u:\R^N\rightarrow\R$, we define the Gagliardo seminorm by setting
\[
[u]_s:=\left(\int_{\R^{2N}}\frac{|u(x)-u(y)|^2}{|x-y|^{N+2s}}dxdy\right)^\frac{1}{2}.
\]
Now, we introduce the fractional Sobolev space (which is a Hilbert space)
 \[
H^s(\R^N)=\{u\in L^2(\R^N):[u]_s<\infty\}, \
\]
with the norm $\|u\|_{H^s}=(\|u\|_{L^2}^2+[u]_s^2)^{\frac{1}{2}}$.
Let the closed subspace
\begin{equation*}
X(\Omega):=\{u\in H^s(\R^N):u=0 \text{  a.e. in } \R^N\backslash\Omega\}.
\end{equation*}
It holds that $X(\Omega)\hookrightarrow L^r(\Omega)$ continuously for $r\in [1, 2^*_s]$ and compactly for $r\in [1, 2^*_s)$, where $2^*_s=\frac{2N}{N-2s}$. Due to the fractional Sobolev inequality, $X(\Omega)$ is a Hilbert space with the inner product given by
\begin{equation*}
\langle u,v\rangle_X:=\int_{\R^{2N}}\frac{(u(x)-u(y)(v(x)-v(y))}{|x-y|^{N+2s}}dxdy,
\end{equation*}
which induces the norm $\|\cdot\|_X=[\cdot]_s$.
We shall denote by $\mu_1$ and  $\mu_2$ the real eigenvalues of the matrix
\begin{equation*}
A:= \left(
 \begin{array}{cc}
 a&b\\
 b&c\\
 \end{array}
 \right)           ,\text{ $a,b,c$ }\in \R.
 \end{equation*}
Without loss of generality, we will assume $\mu_1\leq \mu_2$.  The spectrum of $(-\Delta)^s$, with boundary condition $u=0$
in $\R^N\setminus \Omega$, will be denoted by $\sigma((-\Delta)^s)$,  which consists of the sequence of the
eigenvalues $\{\lambda_{k,s}\}$ satisfying
\begin{equation*}
0<\lambda_{1,s}<\lambda_{2,s}\leq\lambda_{3,s}\leq\ldots \leq \lambda_{j,s}\leq \lambda_{j+1,s}\leq\ldots,
\lambda_{k,s}\rightarrow\infty, \text{ as  } k\rightarrow\infty,
\end{equation*}
and characterized by
\begin{equation*}
\lambda_{1,s}=\inf_{u\in
X(\Omega)\backslash\{0\}}\frac{\int_{\R^{2N}}\frac{|u(x)-u(y)|^2}{|x-y|^{N+2s}}dxdy}{\int_{\R^N}|u(x)|^2dx}
\end{equation*}
and
\begin{equation*}
\lambda_{k+1,s}=\inf_{u\in
\mathbb{P}_{k+1}\backslash\{0\}}\frac{\int_{\R^{2N}}\frac{|u(x)-u(y)|^2}{|x-y|^{N+2s}}dxdy}{\int_{\R^N}|u(x)|^2dx},
\end{equation*}
where
\begin{equation*}
\begin{split}
\mathbb{P}_{k+1}=\{u\in X(\Omega):\langle u,\varphi_{j,s}\rangle_X=0,j=1,2,\ldots,k\},
\end{split}
\end{equation*}
and $\varphi_{k,s}$ denotes the eigenfunction associated to the eigenvalue $\lambda_{k,s}$, for each $k\in \N$.
The following results are true (see \cite{MRVariational}, \cite{MRTheYamabe} and \cite{MR41}).
\begin{enumroman}
\item \label{(1)} If $u\in X(\Omega)$ is a $\lambda_{1,s}$-eigenfunction ($u$ is an eigenfunction corresponding
    to $\lambda_{1,s}$ ), then
either $u(x) > 0$ a.e. in $\Omega$ or $ u(x) <0$ a.e. in $\Omega$;
\item \label{(2)} If $\lambda \in \sigma((-\Delta)^s)\setminus \{\lambda_{1,s}\}$ and $u$ is a
    $\lambda$-eigenfunction, then $u$ changes sign in $\Omega$, and $\lambda$ has finite multiplicity.
\item \label{(3)}  $\varphi_{k,s} \in C^{0,\sigma}(\Omega)$ for some $\sigma\in(0,1)$ and the sequence
    $\{\varphi_{k,s}\}$ is an orthonormal basis in both $L^2(\Omega)$ and $X(\Omega)$.
\end{enumroman}
\begin{Remark}\label{1.1}
For fixed $k\in \N$ we can assume $\lambda_{k,s}<\lambda_{k+1,s}$, otherwise we can suppose that $\lambda_{k,s}$
has multiplicity $l\in \N$, that is
\begin{equation*}
\lambda_{k-1,s}<\lambda_{k,s}=\lambda_{k+1,s}=\ldots=\lambda_{k+l-1,s}<\lambda_{k+l,s},
\end{equation*}
and we denote $\lambda_{k+l,s}=\lambda_{k+1,s}$.
\end{Remark}

In a pioneering paper \cite{MRellipticequations}, Br\'{e}zis and Nirenberg studied the problems of the type
\begin{equation}\label{1.2}
-\Delta u =|u|^{2^*-2}u+\lambda u \text{  in   }\Omega; u=0 \text{  on   } \partial\Omega,
\end{equation}
where $2^*=\frac{N+2}{N-2}$. They proved the existence of nontrivial solutions for  $\lambda >0, N>4$ by
developing some skillful techniques in estimating the Minimax level. This kind of  Br\'{e}zis-Nirenberg problems  have been extensively studied (see, e.g. \cite{MRFortunato,MRComte,MRCostaSilva, MRStruwe1,MRquasilinearX,MRnonsymmetricterm,MRm-Laplacian,MRQuasilinear111,MRTransAmer,MRLower-order,MREABSilva,MRDirichlet,MRZHWei} and references therein). Recently, many well-known Br\'{e}zis-Nirenberg results  in critical local equations  have been extended to  semilinear equations with fractional Laplacian. Specially, we refer to \cite{MRTheYamabe,MRresonantcase,MRBrezisNirenberg,MRBrezisNirenberg2},  where authors studied the following  critical fractional Laplacian problem
\begin{equation}\label{1.3}
(-\Delta)^s u =|u|^{2_s^*-2}u+\lambda u \text{  in   }\Omega; u=0 \text{  in   } \R^N\backslash\Omega,
\end{equation}
and  shown that  problem \eqref{1.3} has a nontrivial weak solution under  the following circumstances:
\begin{enumroman}
\item \label{Theorem 1.2.(1)} $2s<N<4s$ and $\lambda$ is sufficiently large;
\item \label{Theorem 1.2.(2)}  $N=4s$  and $\lambda$ is not an eigenvalue of $(-\Delta)^s u$ in $\Omega$;
\item \label{Theorem 1.2.(2)} $N\geq 4$.
\end{enumroman}
In \cite{MR11}, Gao and Yang  studied the  Br\'{e}zis-Nirenberg type problem  involving the Choquard nonlinearity,  that is
\begin{equation}\label{1.4}
-\Delta u=\lambda u+\left( \int_\Omega\frac{|u|^{2^*_\mu}}{|x-y|^\mu}dy\right)|u|^{2^*_\mu-2}u, \text{ in
}\Omega, u=0,\text{ in }\R^N\backslash\Omega.
\end{equation}
where $\Omega$ is bounded domain in $\R^N$. They proved the
existence, multiplicity and nonexistence results for a range of $\lambda$. Moreover, in \cite{MR88} authors studied a class of critical Choquard equations
\begin{equation}\label{1.5}
-\Delta u=\left( \int_\Omega\frac{|u|^{2^*_\mu}}{|x-y|^\mu}dy\right)|u|^{2^*_\mu-2}u+\lambda f(u), \text{ in } \Omega.
\end{equation}
They proved some existence and multiplicity results for the equation \eqref{1.5} under suitable assumptions  on different types of nonlinearities $f(u)$. In the nonlocal case, Mukherjee and  Sreenadh in \cite{MRmultip} considered  nonlocal counterpart of  problem \eqref{1.4} and obtained  existence, multiplicity and nonexistence results for solutions.


Coming to the system of equations,  elliptic systems involving fractional Laplacian and critical growth nonlinearities have been studied  in \cite{MRNonlinearAnal,MRCriticalBr,MROncriticalsystems,MRspectrum},  extending the Br\'{e}zis and Nirenberg results for variational systems.
Particularly, in \cite{MRspectrum}, Miyagaki and  Pereira studied the following fractional elliptic system
\begin{equation}\label{1.6}
\left\{\begin{aligned}
(-\Delta)^s u & =au+bv+\frac{2p}{p+q}|u|^{p-2}u|v|^q+2\xi_1u|u|^{p+q-2},&&
\text{in } \Omega;\\
(-\Delta)^s v & =bu+cv+\frac{2q}{p+q}|u|^p|v|^{q-2}v+2\xi_2v|v|^{p+q-2},&&
\text{in } \Omega;\\
u &=v=0,\text{ in }   \R^N\setminus\Omega,
\end{aligned}\right.
\end{equation}
extending   \cite{MRCriticalBr}   by means of the Linking Theorem when
\begin{equation*}
\lambda_{k-1,s}\leq\mu_1<\lambda_{k,s}\leq\mu_2<\lambda_{k+1,s}, \text{ if  } k\geq 1.
\end{equation*}
Under theses circumstances,  resonance  and  double resonance phenomena $\lambda_{k-1,s}=\mu_1 $ and  $\lambda_{k,s}=\mu_2$ can occur.
In \cite{MRDoubly},  Giacomoni,  Mukherjee and  Sreenadh discussed the existence and multiplicity of weak solutions for the following  fractional elliptic system involving Choquard  type nonlinearities,
\begin{equation*}
\left\{\begin{aligned}
(-\Delta)^s u & =
\lambda|u|^{q-2}u+\Big(\int_{\Omega} \frac{|v(y)|^{2^*_\mu}}{|x-y|^{\mu}}dy\Big)|u|^{2^*_\mu-2}u,&&
\text{in } \Omega;\\
(-\Delta)^s v & =
\delta|v|^{q-2}v+\Big(\int_{\Omega} \frac{|u(x)|^{2^*_\mu}}{|x-y|^{\mu}}dy\Big)|v|^{2^*_\mu-2}v,&&
\text{in } \Omega;\\
u &=v=0,\text{ in }   \R^N\setminus\Omega,
\end{aligned}\right.
\end{equation*}
where $\lambda,\delta>0$ are real parameters and $1<q<2$.

Motivated by paper \cite{MRspectrum,MRNonlinearAnal}, we discuss the existence and multiplicity results for problem \eqref{1.1} under the conditions   that  (i) $\xi_1=\xi_2=0$,
$1<p,q<2^*_\mu$, (ii) $\xi_1=\xi_2=0$, $p=q=2^*_\mu$, (iii) $\xi_1,\xi_2>0$, $p=q=2^*_\mu$ respectively. The following are the main results.
\begin{theorem}\label{theorem1.2}(Existence I)Assume that $\xi_1=\xi_2=0$, $1<p,q<2^*_\mu$, $b\geq0$ and
$\mu_2<\lambda_{1,s}$. Then problem \eqref{1.1} admits a positive solution.
\end{theorem}
\begin{theorem}\label{theorem1.3}(Existence II) Assume that $\xi_1=\xi_2=0$, $p=q=2^*_\mu$, $b\geq0$  and
$0<\mu_1\leq\mu_2<\lambda_{1,s}$. Then problem  \eqref{1.1} admits a nonnegative solution, provided that either
\begin{enumroman}
\item \label{Theorem 1.2.(1)} $N\geq 4s$ and $\mu_1>0$, or
\item \label{Theorem 1.2.(2)}  $2s<N<4s$ and $\mu_1$ is large enough.
\end{enumroman}
\end{theorem}

\begin{theorem}\label{theorem1.4} (Existence III)
Assume that $\xi_1,\xi_2>0$, $p=q=2^*_\mu$ and
$\lambda_{k-1,s}<\mu_1<\lambda_{k,s}\leq\mu_2<\lambda_{k+1,s}$,  for some $k\in \N$. Then problem \eqref{1.1} admits
a nontrivial solution, if one of the following conditions holds,
\begin{enumroman}
\item \label{Theorem 1.3.(1)} $N\geq4s$ and $\mu_1>0$,
\item \label{Theorem 1.3.(2)} $2s<N<4s$ and $\mu_1$ is large enough.
\end{enumroman}
\end{theorem}

\section{Preliminary Stuff}
\subsection{Notations and setting}
Now,we consider the Hilbert space given by the product space
\begin{equation} \label{2.1}
Y(\Omega):=X(\Omega)\times X(\Omega),
\end{equation}
equipped with the inner product
\begin{equation} \label{2.2}
\langle(u,v),(\varphi,\psi)\rangle_Y:=\langle u,\varphi\rangle_X+\langle v,\psi\rangle_X
\end{equation}
and the norm
\begin{equation} \label{2.3}
\|(u,v)\|_Y:=(\|u\|^2_X+\|v\|^2_X)^{\frac{1}{2}}.
\end{equation}
We shall consider $L^m(\Omega)\times L^m(\Omega)$($m>1$) equipped with the standard product norm
\begin{equation} \label{2.4}
\|(u,v)\|_{L^m\times L^m}:=(\|u\|^2_{L^m} +\|v\|^2_{L^m})^{\frac{1}{2}}.
\end{equation}
We recall that
\begin{equation} \label{2.5}
\mu_1|U|^2\leq(AU,U)_{R^2}\leq\mu_2|U|^2,\text{  for all  }U:=(u,v)\in \R^2.
\end{equation}
By a solution of \eqref{1.1} we mean a weak solution, that is, a pair of functions $(u,v)\in Y(\Omega)$ such
that
\begin{equation*}
\langle(u,v),(\varphi,\psi)\rangle_Y-\int_\Omega (A(u,v),(\varphi,\psi))_{\R^2}dx-\int_\Omega \frac{\partial
F}{\partial u}\varphi dx-\int_\Omega \frac{\partial F}{\partial v}\psi dx=0,
\end{equation*}
for all $(\varphi,\psi)\in Y(\Omega)$, where
\begin{equation}\label{2.6}
F(u,v)=\frac{2}{p+q}\int_{\Omega}\frac{|v(y)|^{q}}{|x-y|^\mu}dy|u|^{p}+\frac{1}{2^*_\mu}\left[\xi_1 \int_{\Omega}\frac{|u(y)|^{2^*_\mu}}{|x-y|^\mu}dy|u|^{2^*_\mu}+\xi_2\int_{\Omega}\frac{|v(y)|^{2^*_\mu}}{|x-y|^\mu}dy|v|^{2^*_\mu}\right].
\end{equation}
Now define the functional $J_s:Y(\Omega)\rightarrow \R$ by setting
\begin{eqnarray*}
\begin{aligned}
J_s(U)\equiv J_s(u,v) = &\frac{1}{2}\int_{\R^{2N}}\frac{|u(x)-u(y)|^2+|v(x)-v(y)|^2}{|x-y|^{N+2s}}dxdy \\
&-\frac{1}{2}\int_{\R^N}(A(u,v),(u,v))_{\R^2}dx-\int_{\Omega}F(U)dx,
\end{aligned}
\end{eqnarray*}
whose Fr\'{e}chet derivative is given by
\begin{eqnarray*}
\begin{aligned}
J'_s(u,v)(\varphi,\psi)=&\int_{\R^{2N}}\frac{(u(x)-u(y))(\varphi(x)-\varphi(y))+(v(x)-v(y))(\psi(x)-\psi(y))}{|x-y|^{N+2s}}dxdy\\
&-\int_{\Omega}(A(u,v),(\varphi,\psi))_{\R^2}dx-\frac{2p}{p+q}\int_{\Omega}\frac{|u(x)|^{p-2}u(x)|v(y)|^{q}}{|x-y|^\mu}\varphi dxdy
\\&-\frac{2q}{p+q}\int_{\Omega}\frac{|u(x)|^{p}|v(y)|^{q-2}v(y)}{|x-y|^\mu}\psi dxdy-2\xi_1\int_{\Omega}\frac{|u(x)|^{2^*_\mu-2}u(x)|u(y)|^{2^*_\mu}}{|x-y|^\mu}\varphi dxdy\\
&-2\xi_2 \int_{\Omega}\frac{|v(x)|^{2^*_\mu}|v(y)|^{2^*_\mu-2}v(y)}{|x-y|^\mu}\psi dxdy,\\
\end{aligned}
\end{eqnarray*}
for every $(\varphi,\psi)\in Y(\Omega)$.\\

In this paper, we set the following notation for product space $S\times S:=S^2$ and
\begin{equation} \label{2.7}
w^+(x):=\text{ max }\{w(x),0\},w^-(x):=\text{ min }\{w(x),0\},
\end{equation}
for positive and negative part of a function $w$. Consequently we get $w=w^++w^-$. During chains of
inequalities, universal constants will be denoted by the
same letter $C$ even if their numerical value may change from line to line.
\subsection{Some important conclusions}
Here we list some important conclusions.
The first one is the following well-known Hardy-Littlewood-Sobolev inequality.
\begin{proposition}(Hardy-Littlewood-Sobolev inequality, \cite[Theorem 4.3]{MR19})\label{Proposition 2.1}
Let $t,r>1$ and $0<\mu<N$ with $\frac{1}{t}+\frac{\mu}{N}+\frac{1}{r}=2, f\in L^t(\R^N)$ and $h\in L^r (\R^N)$.
There exists a sharp constant $C(t,N,\mu,r)$, independent of $f,h$ such that
\begin{equation}\label{2.8}
\int_{\R^N}\int_{\R^N} \frac {f(x)h(y)}{|x-y|^{\mu}}dxdy \leq C(t,N,\mu,r)\|f\|_{L^t(\R^N)}\|h\|_{L^r(\R^N)}.
\end{equation}
if $t=r=\frac{2N}{2N-\mu}$ then
\begin{equation*}
C(t,N,\mu,r)=C(N,\mu)=\pi^{\frac{\mu}{2}}\frac{\Gamma(\frac{N}{2}-\frac{\mu}{2})}{\Gamma(N-\frac{\mu}{2})}\left\{\frac{\Gamma(\frac{N}{2})}{\Gamma(N)}\right\}^{-1+\frac{\mu}{N}}.
\end{equation*}
In this case, there is equality in \eqref{2.8} if and only if $f\equiv(constant)h$ and
\begin{equation*}
h(x)=A(\gamma^2+|x-a|^2)^{\frac{-(2N-\mu)}{2}}
\end{equation*}
for some $A\in \mathcal{C} ,0 \neq\gamma\in{\R} $ and $a\in {\R ^N}$.
\end{proposition}
\begin{Remark}\label{remark2.2}
For $u\in{H^s(\R^N)}$, let $f=h=|u|^p$,  by Hardy-Littlewood-Sobolev inequality,
\begin{equation*}
\int_{\R^N}\int_{\R^N} \frac {|u(x)|^p|u(y)|^p}{|x-y|^{\mu}}dxdy
\end{equation*}
is well defined for all $p$ satisfying
\begin{equation*}
2_\mu:=\left(\frac{2N-\mu}{N}  \right)\leq p\leq\left( \frac{2N-\mu}{N-2s}\right):=2^*_\mu.
\end{equation*}
\end{Remark}
Next result is a basic inequality, which plays a great role in the latter proof.
\begin{proposition}(  \cite[Lemma 2.3]{MRDoubly}) \label{proposition2.3}
For $u,v\in L^{\frac{2N}{2N-\mu}}(\R^N)$, we have
\begin{equation*}
\int_{\R^N}\int_{\R^N}\frac{|u(x)|^p|v(y)|^p}{|x-y|^\mu}dxdy\leq\left(\int_{\R^N}\int_{\R^N}\frac{|u(x)|^p|u(y)|^p}{|x-y|^\mu}dxdy\right)^{\frac{1}{2}}\left(\int_{\R^N}\int_{\R^N}\frac{|v(x)|^p|v(y)|^p}{|x-y|^\mu}dxdy\right)^{\frac{1}{2}},
\end{equation*}
where $\mu\in(0,N)$ and $p\in[2_\mu,2_\mu^*]$.
\end{proposition}

\subsection{Abstract critical point theorems}
We will prove Theorem \ref{theorem1.3} and Theorem \ref{theorem1.4} by the following abstract critical point theorems  respectively.

\begin{theorem}(Mountain pass theorem, \cite[Theorem 2.10]{MRMinimaxTheorem})
Let $X$ be a Banach  space, $J\in C^1(X, \R)$, $e\in X$ and $r>0$ be such that $\|e\|>r$ and
\begin{equation*}
b:=\inf_{\|u\|=r}J(u)>J(0)\geq J(e).
\end{equation*}
If $J$ satisfies the $(PS)_c$ condition with
\begin{equation*}
\begin{split}
c:=\inf_{\gamma\in \Gamma}\max_{t\in [0,1]}J(\gamma(t)),
\end{split}
\end{equation*}
\begin{equation*}
\begin{split}
\Gamma:=\{\gamma \in C([0,1],X):\gamma(0)=0, \gamma(1)=e\}.
\end{split}
\end{equation*}
Then $c$ is a critical value of $J$.
\end{theorem}

\begin{theorem}(Linking theorem, \cite[Theorem 2.12]{MRMinimaxTheorem}) Let $X$ be a real Banach space with $X=V\oplus W$,
where $V$ is finite dimensional. Suppose $J\in C^1(X,\R)$ and
\begin{enumroman}
\item \label{lemma 5.2.a} There are constants $\rho,\alpha>0$ such that $J|_{\partial B_\rho\bigcap W}\geq
    \alpha$, and
\item \label{lemma 5.2.b} There is an $e\in \partial B_\rho\bigcap W $ and constants $R_1,R_2>\rho$ such that
    $J|_{\partial Q}\leq0$, where
\end{enumroman}
\begin{eqnarray*}
\begin{aligned}
Q=(\overline{B_{R_1}}\bigcap V)\oplus \{re,0<r<R_2\}.
\end{aligned}
\end{eqnarray*}
Then $J$ possesses a $(PS)_c$  sequence where $c \geq\alpha$ can be characterized as
\begin{eqnarray*}
\begin{aligned}
c=\inf_{h\in\Gamma}\max_{u\in Q}J(h(u)),
\end{aligned}
\end{eqnarray*}
where
\begin{eqnarray*}
\begin{aligned}
\Gamma=\{h\in C(\overline{Q},X):h=id \text{ on } \partial Q\}.
\end{aligned}
\end{eqnarray*}
\end{theorem}
\begin{Remark}
Here $\partial Q$ is  the boundary of $Q$  relative to the space $V\oplus \text{span}\{e\}$, and when $V=\{0\}$,
this theorem refers to the usual mountain pass Theorem. We recall that if $J|_{V}\leq 0$ and $J(u)\leq 0, \forall
u\in V\oplus \text{span}\{e\} $ with $\|u\|\geq R$,  then $J$ verifies $(ii)$ for $R$ large enough.
Fixed $k \in \N$, define the following subspaces
\begin{eqnarray*}
\begin{aligned}
V=\text{ span
}\{(0,\varphi_{1,s}),(\varphi_{1,s},0),(0,\varphi_{2,s}),(\varphi_{2,s},0),\ldots,(0,\varphi_{k-1,s}),(\varphi_{k-1,s},0)\}
\end{aligned}
\end{eqnarray*}
and
\begin{eqnarray*}
\begin{aligned}
W=V^\bot=(\mathbb{P}_k)^2.
\end{aligned}
\end{eqnarray*}
\end{Remark}
\section{$Case~1:\xi_1=\xi_2=0$, $1<p,q<2^*_\mu$}

\noindent\textbf{Proof of Theorem 1.2}
Let $\Omega$ be a bounded domain and suppose that
\begin{equation}
\begin{split}\label{3.1}
b\geq 0,
\end{split}
\end{equation}
\begin{equation}
\begin{split}\label{3.2}
\mu_2 < \lambda_{1,s}.
\end{split}
\end{equation}
Consider the function $I:Y(\Omega)\rightarrow\R $ defined by
\begin{equation}
\begin{split}\label{3.3}
I(U):=\frac{1}{2}\|U\|_Y^2-\frac{1}{2}\int_\Omega(AU,U)_{\R^2}dx.
\end{split}
\end{equation}
We shall minimize the functional $I$ restricted to the set
\begin{equation*}
M:=\{U=(u,v)\in Y(\Omega):\int_\Omega\int_\Omega\frac{|u^+(x)|^p|v^+(y)|^q}{|x-y|^\mu}dxdy=1 \}.
\end{equation*}
By virtue of \eqref{3.2} the embedding $X(\Omega)\hookrightarrow L^2(\Omega)$(with the sharp constant
$\lambda_{1,s}$), we have
\begin{equation}
\begin{split}\label{3.4}
I(U)\geq\frac{1}{2}\text{min}\left\{1,(1-\frac{\mu_2}{\lambda_{1,s}})\right\}\|U\|_Y^2\geq 0.
\end{split}
\end{equation}
Define
\begin{equation}
\begin{split}\label{3.5}
I_0:=\inf_{\M}\ I,
\end{split}
\end{equation}
and let $(U_n)=(u_n,v_n)\subset \M $ be a minimizing sequence for $I_0$. Then $I(U_n)=I_0+o_n(1)\leq C$, for
some $C>0$ (where $o_n(1)\rightarrow 0$, as $n\rightarrow\infty$) and consequently by \eqref{3.4}, we get
\begin{equation}
\begin{split}\label{3.6}
[u_n]_s^2+[v_n]_s^2=\|u_n\|^2_X+\|v_n\|^2_X=\|U_n\|^2_Y\leq C'.
\end{split}
\end{equation}
Hence, there are two subsequences of $\{u_n\}\subset X(\Omega)$ and $\{v_n\}\subset X(\Omega)$ (that we will
still label as $u_n$ and $v_n$) such that $U_n=(u_n,v_n)$ converges to some $U=(u,v)$ in $Y(\Omega)$ weakly and
\begin{equation}
\begin{split}\label{3.7}
[u]_s^2\leq\liminf_{n}\ \int_{\R^{2N}}\frac{|u_n(x)-u_n(y)|^2}{|x-y|^{N+2s}}dxdy,
\end{split}
\end{equation}
\begin{equation}
\begin{split}\label{3.8}
[v]_s^2\leq\liminf_{n}\ \int_{\R^{2N}}\frac{|v_n(x)-v_n(y)|^2}{|x-y|^{N+2s}}dxdy,
\end{split}
\end{equation}
Now we will show that $U:=(u,v)\in \M $. Indeed, since $(U_n)\subset \M$, we have
\begin{equation}
\begin{split}\label{3.9}
\int_\Omega\int_\Omega\frac{|u^+_n(x)|^p|v^+_n(y)|^q}{|x-y|^\mu}dxdy=1.
\end{split}
\end{equation}
In view of the compact embedding $X(\Omega)\hookrightarrow L^r(\Omega)$ for all
$r<2^*_s=\frac{2N}{N-2s}$, as $1<p,q<2^*_\mu$, we get
\begin{equation}
\begin{split}\label{3.10}
\int_\Omega\int_\Omega\frac{|u_n^+(x)|^p|v_n^+(y)|^q}{|x-y|^\mu}dxdy\rightarrow\int_\Omega\int_\Omega\frac{|u^+(x)|^p|v^+(y)|^q}{|x-y|^\mu}dxdy,
\text{    as }n\rightarrow\infty,
\end{split}
\end{equation}
thus $\int_\Omega\int_\Omega\frac{|u^+(x)|^p|v^+(y)|^q}{|x-y|^\mu}dxdy=1$ and consequently $ U:=(u,v)\in \M $ with $u,v\neq0$. We now
show that $U=(u,v)$ is a minimizer for $I$ on $\M$ and both  components $u,v$ are nonnegative. By passing to
the
limit in $I(U_n)=I_0+o_n(1)$, where  $o_n(1)\rightarrow0$ as $n\rightarrow\infty$, using \eqref{3.7},
\eqref{3.8} and the strong convergence of $(u_n,v_n)$ to $(u,v)$ in $(L^2(\Omega))^2$, as $n\rightarrow\infty$,
we conclude that $I(U)\leq I_0$.
Moreover, since $U\in \M $ and $I_0=\inf_{\M}\ I\leq I(U)$, we achieve that $I(U)=I_0$. This proves the
minimality of $U\in \M $. On the other hand, let
\[
G(U)=\int_\Omega\int_\Omega\frac{|u^+(x)|^p|v^+(y)|^q}{|x-y|^\mu}dxdy-1,
\]
where $U(u,v)\in Y(\Omega)$. Note that $G\in C^1 $ and since $U\in \M$,
\[
G'(U)U=(p+q)\int_\Omega\int_\Omega\frac{|u^+(x)|^p|v^+(y)|^q}{|x-y|^\mu}dxdy=p+q\neq 0,
\]
hence, by Lagrange Multiplier Theorem, there exists a multiplier $\zeta\in \R $ such that
\begin{equation}
\begin{split}\label{3.11}
I'(U)(\varphi,\psi)=\zeta G'(U)(\varphi,\psi),\forall (\varphi,\psi)\in Y(\Omega).
\end{split}
\end{equation}
Taking $(\varphi,\psi)=(u^-,v^-):=U^-$in \eqref{3.11}, we get
\begin{equation*}
\begin{split} \|U^-\|^2_Y=&\int_{\R^{2N}}\frac{u^+(x)u^-(y)+u^-(x)u^+(y)}{|x-y|^{N+2s}}dxdy\\[10pt]
& +\int_{\R^{2N}}\frac{v^+(x)v^-(y)+v^-(x)v^+(y)}{|x-y|^{N+2s}}dxdy\\[10pt]
& +\int_{\Omega}(AU,U^-)_{\R^{2}}dx.\\[10pt]
\end{split}
\end{equation*}
Dropping this formula into the expression of $I(U^-)$, we have
\begin{equation}
\begin{split}\label{3.12}
I(U^-)=&\frac{b}{2}\int_{\Omega}(v^+u^-+u^+v^-)dx+\frac{1}{2}\int_{\R^{2N}}\frac{u^+(x)u^-(y)+u^-(x)u^+(y)}{|x-y|^{N+2s}}dxdy\\
&+\frac{1}{2}\int_{\R^{2N}}\frac{v^+(x)v^-(y)+v^-(x)v^+(y)}{|x-y|^{N+2s}}dxdy\leq0,\\
\end{split}
\end{equation}
since $b\geq0,u^-\leq0$ and $u^+\geq0$. On the other hand,
\begin{equation*}
\begin{split}
I(U^-)\geq \frac{1}{2}\text{min}\left\{1,(1-\frac{\mu_2}{\lambda_{1,s}})\right\}\|U^-\|^2_Y\geq 0,
\end{split}
\end{equation*}
we get $U^-=(u^-,v^-)=(0,0)$ and therefore $u,v\geq 0$. We now prove the
existence of a positive solution to \eqref{1.1}. Using again \eqref{3.11}, we see that
\[
\|U\|^2_Y-\int_{\Omega}(AU,U)_{\R^2}dx-\zeta(p+q)=0
\]
and since $U\in \M$, we conclude that
\[
I_0=I(U)=\frac{\zeta(p+q)}{2}>0,
\]
 Then by \eqref{3.11}, $U$ satisfies the following system, weakly,
\begin{equation*}
\left\{\begin{aligned}
&(-\bigtriangleup)^s u
=au+bv+\frac{2pI_0}{p+q}\int_{\Omega}\int_{\Omega}\frac{|u|^{p-1}|v|^q}{|x-y|^{\mu}}dxdy,&&\text{ in }
\Omega;\\
&(-\bigtriangleup)^s v
=bu+cv+\frac{2qI_0}{p+q}\int_{\Omega}\int_{\Omega}\frac{|u|^p|v|^{q-1}}{|x-y|^{\mu}}dxdy,&&\text{ in }
\Omega;\\
&u=v=0, &&\text{ in } \R^N\backslash \Omega. \\
\end{aligned}\right.
\end{equation*}
Now using the homogeneity of system, we get $\tau > 0$ such that $W=(I_0)^\tau U$ is a solution of \eqref{1.1}.
Since $b\geq0$ and $u,v\geq0$ we get, in weak sense
\begin{equation*}
\left\{\begin{aligned}
 &(-\bigtriangleup)^s u \geq au,&&\text{ in } \Omega;\\
& (-\bigtriangleup)^s v \geq cv,&&\text{ in } \Omega;\\
&u\geq 0,v\geq0 &&\text{ in } \Omega;\\
&u=v=0, &&\text{ in } \R^N\backslash \Omega. \\
\end{aligned}\right.
\end{equation*}
By the strong maximum principle(cf. \cite{MR9}, Theorem 2.5), we conclude $u,v>0$ in $\Omega$.

\section{$Case~2:\xi_1=\xi_2=0, p=q=2^*_\mu.$}
In this case,  we have the  function $J_s:Y(\Omega)\rightarrow \R$ by setting
\begin{eqnarray*}
\begin{aligned}
J_s(U)\equiv J_s(u,v) = &\frac{1}{2}\int_{\R^{2N}}\frac{|u(x)-u(y)|^2+|v(x)-v(y)|^2}{|x-y|^{N+2s}}dxdy \\
&-\frac{1}{2}\int_{\R^N}(A(u,v),(u,v))_{\R^2}dx-\frac{1}{2^*_\mu}\int_{\Omega}\int_{\Omega}\frac{|u^+(x)|^{2^*_\mu}|v^+(y)|^{2^*_\mu}}{|x-y|^{\mu}}dxdy,
\end{aligned}
\end{eqnarray*}
whose Fr\'{e}chet derivative is given by
\begin{eqnarray}\label{4.13}
\begin{aligned}
J'_s(u,v)(\varphi,\psi)=&\int_{\R^{2N}}\frac{(u(x)-u(y))(\varphi(x)-\varphi(y))+(v(x)-v(y))(\psi(x)-\psi(y))}{|x-y|^{N+2s}}dxdy\\
&-\int_{\Omega}(A(u,v),(\varphi,\psi))_{\R^2}dx-\int_{\Omega}\int_{\Omega}\frac{|u^+(x)|^{2^*_\mu-1}|v^+(y)|^{2^*_\mu}}{|x-y|^\mu}\varphi dxdy
\\&-\int_{\Omega}\int_{\Omega}\frac{|u^+(x)|^{2^*_\mu}|v^+(y)|^{2^*_\mu-1}}{|x-y|^\mu}\psi dxdy,\\
\end{aligned}
\end{eqnarray}
for every $(\varphi,\psi)\in Y(\Omega)$.
\subsection{Minimizers and some estimates}
Let
\begin{equation} \label{4.1}
S_s:=\inf_{u\in X(\Omega)\backslash\{0\}} S_s(u),
\end{equation}
where
\begin{equation} \label{4.2}
S_s(u):=\frac{\int_{\R^{2N}}\frac{|u(x)-u(y)|^2}{|x-y|^{N+2s}}dxdy}{(\int_{\R^N}|u(x)|^{2^*_s}dx)^\frac{2}{2^*_s}}
\end{equation}
is the associated Rayleigh quotient. Define the following related minimizing problems as
\begin{equation} \label{4.3}
S_s^H=\inf_{u\in
X(\Omega)\backslash\{0\}}\frac{\int_{\R^{2N}}\frac{|u(x)-u(y)|^2}{|x-y|^{N+2s}}dxdy}{(\int_{\Omega}\int_{\Omega}\frac{|u(x)|^{2^*_\mu}|u(y)|^{2^*_\mu}}{|x-y|^\mu}dxdy)^{\frac{1}{2^*_\mu}}}
\end{equation}
and
\begin{equation} \label{4.4}
\widetilde{S}_s^H=\inf_{(u,v)\in
Y(\Omega)\backslash\{(0,0)\}}\frac{\int_{\R^{2N}}\frac{|u(x)-u(y)|^2+|v(x)-v(y)|^2}{|x-y|^{N+2s}}dxdy}{(\int_{\Omega}\int_{\Omega}\frac{|u(x)|^{2^*_\mu}|v(y)|^{2^*_\mu}}{|x-y|^\mu}dxdy)^{\frac{1}{2^*_\mu}}}.
\end{equation}

\begin{proposition}\label{proposition4.1}\
\begin{enumroman}
\item \label{Proposition 4.1.1}(\cite[Lemma 2.15]{MRDAvenia})
The constant $S_s^H $ is achieved by $u$ if and only if $u$ is of the form
\begin{equation*}
C(\frac{t}{t^2+|x-x_0|^2})^{\frac{N-2s}{2}},x\in \R^N,
\end{equation*}
for some $x_0\in \R^N,C>0$ and  $t>0$. Also it satisfies
\begin{equation}\label{4.5}
(-\Delta)^s u=\left(\int_{\R^N}\frac{|u|^{2^*_\mu}}{|x-y|^\mu}dy\right)|u|^{2^*_\mu-2}u \text{  in  } \R^N.
\end{equation}
and  this characterization of $u$ also provides the minimizers for $S_s$.
\item \label{Proposition 4.1.2}(\cite[Lemma 2.5]{MRDoubly})
 $S_s^H=\frac{S_s}{C(N,\mu)^{\frac{1}{2^*_\mu}}}$.
\item \label{Proposition 4.1.3}(\cite[Lemma 2.6]{MRDoubly})
 $\widetilde{S}_s^H=2S_s^H$.
\end{enumroman}
\end{proposition}

Now  we will construct auxiliary functions and make some estimates with the help of Proposition \ref{proposition4.1}.  From \cite{MRBrezisNirenberg},
consider the family of function $\{U_\epsilon\}$ defined as
\begin{equation*}
 U_\epsilon(x)=\epsilon^{-\frac{(N-2s)}{2}}u^*(\frac{x}{\epsilon}),x\in \R^N,
\end{equation*}
where $u^{*}(x)=\overline{u}\left(\frac{x}{S_s^{\frac{1}{2s}}}\right)$, $\overline{u}(x)=\frac{\widetilde{u}(x)}{\|\widetilde{u}\|_{L^{2^*_s}}}$ and $\widetilde{u}=\alpha(\beta^2+|x|^2)^{-\frac{N-2s}{2}}$ with $\alpha \in \R\backslash \{0\}$ and $\beta>0$ are fixed constants. Then for each $\epsilon >0$, $U_\epsilon$ satisfies
\begin{equation*}
(-\Delta)^su=|u|^{2^*_s-2}u  ,\text{ in  } \R^N,
\end{equation*}
in addition,
\begin{equation*}
\int_{\R^N}\int_{\R^N}\frac{|U_\epsilon(x)-U_\epsilon(y)|^2}{|x-y|^{N+2s}}dxdy=\int_{\R^N}|U_\epsilon|^{2^*_s}dx=S_s^{\frac{N}{2s}}.
\end{equation*}

Without loss of generality, we assume $0\in{\Omega}$ and fix $\delta > 0$ such that $B_{4\delta}\subset \Omega
$. Let $\eta \in C^\infty (\R^N)$ be such that $0\leq\eta\leq 1$ in $\R^N$, $\eta\equiv1$ in $B_\delta$ and
$\eta\equiv 0$ in $\R^N\backslash B_{2\delta}$. For $\epsilon>0$, we denote by $u_\epsilon$ the following
function
\begin{equation*}
u_\epsilon(x)=\eta(x)U_\epsilon(x),
\end{equation*}
for $x\in \R^N$. We have the following results for $u_\epsilon$ in \cite[Propositions 21, Propositions 22]{MRBrezisNirenberg} and \cite[Proposition 7.2]{MRTheYamabe}.
\begin{proposition}\label{Proposition4.2}
Let $s\in(0,1)$ and $N>2s$. Then, the following estimates hold true as $\epsilon\rightarrow0$:
\begin{enumroman}
\item \label{Proposition 4.3.a}
    $\int_{\R^{2N}}\frac{|u_\epsilon(x)-u_\epsilon(y)|^2}{|x-y|^{N+2s}}dxdy\leq
    S^{\frac{N}{2s}}_s+O(\epsilon^{N-2s})$,
\item \label{Proposition 4.3.b}$\int_{\R^{N}}|u_\epsilon|^{2^*_s}dx =S^{\frac{N}{2s}}_s+O(\epsilon^N)$,
\item \label{Proposition 4.3.c}$\int_{\R^{N}}|u_\epsilon|^2dx\geq
\begin{cases}
C_s\epsilon^{2s}+O(\epsilon^{N-2s}),& if~N>4s;\\
C_s\epsilon^{2s}|log\epsilon|+O(\epsilon^{2s}), & if~N=4s;\\
C_s\epsilon^{N-2s}+O(\epsilon^{2s}), & if~2s<N<4s;\\
\end{cases}$\\
for some positive constant $C_s$ depending on $s$.
\item \label{Proposition 4.3.d}$\int_{\R^{N}}|u_\epsilon|dx=O(\epsilon^{\frac{N-2s}{2}})$.
\end{enumroman}
\end{proposition}
\begin{Remark}\label{Remark4.3}
From  Proposition \ref{proposition4.1} (ii) and Proposition \ref{Proposition4.2} (i), we get
\begin{equation}\label{4.6}
\int_{\R^{2N}}\frac{|u_\epsilon(x)-u_\epsilon(y)|^2}{|x-y|^{N+2s}}dxdy\leq
S^{\frac{N}{2s}}_s+O(\epsilon^{N-2s})=C(N,\mu)^{\frac{N-2s}{2N-\mu}\cdot\frac{N}{2s}}(S_s^H)^{\frac{N}{2s}}+O(\epsilon^{N-2s}).
\end{equation}
\end{Remark}
\begin{proposition}(\cite[Proposition 2.8]{MRDoubly})\label{Proposition4.4}~
Let $s\in(0,1)$ and $N>2s$. Then, the following estimate holds true as $\epsilon\rightarrow0$:
\begin{equation}\label{4.7}
\int_{\Omega}\int_{\Omega}\frac{|u_\epsilon(x)|^{2^*_\mu}|u_\epsilon(y)|^{2^*_\mu}}{|x-y|^{\mu}}dxdy\geq
C(N,\mu)^{\frac{N}{2s}}(S_s^H)^{\frac{2N-\mu}{2s}}-O(\epsilon^{N}).
\end{equation}
\end{proposition}
Now consider the following minimization problem
\begin{equation*}
\begin{split}
S_{s,\lambda}=\inf_{v\in X(\Omega)\backslash \{0\}}S_{s,\lambda}(v),
\end{split}
\end{equation*}
where
\begin{equation*}
\begin{split}
S_{s,\lambda}(v)=\frac{\int_{\R^{2N}}\frac{|v(x)-v(y)|^2}{|x-y|^{N+2s}}dxdy-\lambda\int_{\R^N}|v(x)|^2dx}{\left(\int_{\Omega}\int_{\Omega}\frac{|v(x)|^{2^*_\mu}|v(y)|^{2^*_\mu}}{|x-y|^{\mu}}dxdy\right)^\frac{1}{2^*_\mu}}.
\end{split}
\end{equation*}
\begin{lemma}\label{lemma4.5} Let $N>2s$ and $s\in(0,1)$. Then the following facts hold.
\begin{enumroman}
\item \label{proposition4.5.a} For $N\geq 4s$,
\begin{equation*}
S_{s,\lambda}(u_\epsilon)<S^H_s, \text{ for all } \lambda>0, \text{  provided } \epsilon>0 \text{ is
sufficiently small }.
\end{equation*}
\item \label{proposition4.5.b} For $2s<N<4s$,   there exists $\lambda_s>0$ such that for all
    $\lambda>\lambda_s$, we have
\begin{equation*}
\begin{split}
S_{s,\lambda}(u_\epsilon)<S^H_s,  \text{ provided } \epsilon>0 \text{ is sufficiently small }.
\end{split}
\end{equation*}
\end{enumroman}
\end{lemma}
\begin{proof}
Case~1: $N>4s$. By Proposition \ref{Proposition4.2} (iii), \eqref{4.6} and \eqref{4.7}, we infer
\begin{equation*}
\begin{split}
S_{s,\lambda}(u_\epsilon)&\leq \frac{C(N,\mu)^{\frac{N-2s}{2N-\mu}\cdot\frac{N}{2s}}(S_s^H)^{\frac{N}{2s}}+O(\epsilon^{N-2s})-\lambda
C_s\epsilon^{2s}+O(\epsilon^{N-2s})}{\left(C(N,\mu)^{\frac{N}{2s}}(S_s^H)^{\frac{2N-\mu}{2s}}-O(\epsilon^{N})\right)^{\frac{1}{2^*_\mu}}}\\
&\leq S^H_s -\lambda C_s\epsilon^{2s}+O(\epsilon^{N-2s})\\
&< S_s^H, \text{  if $\lambda>0$, $\epsilon>0$ is sufficiently small.}
\end{split}
\end{equation*}
Case~2: $N=4s$.
\begin{equation*}
\begin{split}
S_{s,\lambda}(u_\epsilon)&\leq \frac{C(N,\mu)^{\frac{N-2s}{2N-\mu}\cdot\frac{N}{2s}}(S_s^H)^{\frac{N}{2s}}+O(\epsilon^{N-2s})-\lambda
C_s\epsilon^{2s}|log\epsilon|+O(\epsilon^{2s})}{\left(C(N,\mu)^{\frac{N}{2s}}(S_s^H)^{\frac{2N-\mu}{2s}}-O(\epsilon^{N})\right)^{\frac{1}{2^*_\mu}}}\\
&\leq S_s^H -\lambda C_s\epsilon^{2s}|log\epsilon|+O(\epsilon^{2s})\\
&< S_s^H, \text{  if $\lambda>0$, $\epsilon>0$ is sufficiently small.}
\end{split}
\end{equation*}
Case~3: $2s<N<4s$.
\begin{equation*}
\begin{split}
S_{s,\lambda}(u_\epsilon)&\leq \frac{C(N,\mu)^{\frac{N-2s}{2N-\mu}\cdot\frac{N}{2s}}(S_s^H)^{\frac{N}{2s}}+O(\epsilon^{N-2s})-\lambda
C_s\epsilon^{N-2s}+O(\epsilon^{2s})}{\left(C(N,\mu)^{\frac{N}{2s}}(S_s^H)^{\frac{2N-\mu}{2s}}-O(\epsilon^{N})\right)^{\frac{1}{2^*_\mu}}}\\
&\leq S_s^H +\epsilon^{N-2s}(O(1)-\lambda C_s)+O(\epsilon^{2s}),\\
&< S_s^H,
\end{split}
\end{equation*}
for all $\lambda >0 $ large enough ($\lambda\geq \lambda_s$), $\epsilon>0$ sufficiently small.
\end{proof}

\subsection{Compactness convergence}
\begin{proposition}\label{proposition4.6}
Let $s\in(0,1),N>2s$ and $0<\mu<N$. If $\{u_n\},\{v_n\}$ are  bounded sequences in
$L^{\frac{2N}{N-2s}}(\Omega)$ such that $u_n\rightarrow u,v_n\rightarrow v$ almost everywhere in $\Omega$ as
$n\rightarrow\infty$, we have
\begin{eqnarray*}
\begin{aligned}
&\int_{\Omega}\int_{\Omega}\frac{|u_n(x)|^{2^*_\mu}|v_n(y)|^{2^*_\mu}}{|x-y|^{\mu}}dxdy-\int_{\Omega}\int_{\Omega}\frac{|(u_n-u)(x)|^{2^*_\mu}|(v_n-v)(y)|^{2^*_\mu}}{|x-y|^{\mu}}dxdy\\
&\rightarrow \int_{\Omega}\int_{\Omega}\frac{|u(x)|^{2^*_\mu}|v(y)|^{2^*_\mu}}{|x-y|^{\mu}}dxdy,
\end{aligned}
\end{eqnarray*}
as $n\rightarrow\infty$.
\end{proposition}
\begin{proof}
First, similarly to the proof of the Br\'{e}zis-Lieb Lemma in \cite{MRBrezis}, we know that
\begin{eqnarray}\label{4.8}
\begin{aligned}
|u_n|^{2^*_\mu}-|u_n-u|^{2^*_\mu}\rightharpoonup |u|^{2^*_\mu},
\end{aligned}
\end{eqnarray}
\begin{eqnarray}\label{4.9}
\begin{aligned}
|v_n|^{2^*_\mu}-|v_n-v|^{2^*_\mu}\rightharpoonup |u|^{2^*_\mu},
\end{aligned}
\end{eqnarray}
in $L^{\frac{2N}{2N-\mu}}(\Omega)$ as $n\rightarrow\infty$. The Hardy-Littlewood-Sobolev inequality implies that
\begin{eqnarray}\label{4.10}
\begin{aligned}
\int_{\Omega}\frac{|u_n(y)|^{2^*_\mu}-|(u_n-u)(y)|^{2^*_\mu}}{|x-y|^{\mu}}dy\rightharpoonup \int_{\Omega}\frac{|u(y)|^{2^*_\mu}}{|x-y|^{\mu}}dy
\end{aligned}
\end{eqnarray}
\begin{eqnarray}\label{4.11}
\begin{aligned}
\int_{\Omega}\frac{|v_n(y)|^{2^*_\mu}-|(v_n-v)(y)|^{2^*_\mu}}{|x-y|^{\mu}}dy\rightharpoonup \int_{\Omega}\frac{|v(y)|^{2^*_\mu}}{|x-y|^{\mu}}dy
\end{aligned}
\end{eqnarray}
in $L^{\frac{2N}{\mu}}(\Omega)$ as $n\rightarrow\infty$. On the other hand, we notice that
\begin{eqnarray}\label{4.12}
\begin{aligned}
&\int_{\Omega}\int_{\Omega}\frac{|u_n(x)|^{2^*_\mu}|v_n(y)|^{2^*_\mu}}{|x-y|^{\mu}}dxdy-\int_{\Omega}\int_{\Omega}\frac{|(u_n-u)(x)|^{2^*_\mu}|(v_n-v)(y)|^{2^*_\mu}}{|x-y|^{\mu}}dxdy\\
&=\int_{\Omega}\int_{\Omega}\frac{(|u_n(x)|^{2^*_\mu}-|(u_n-u)(x)|^{2^*_\mu})(|v_n(y)|^{2^*_\mu}-|(v_n-v)(y)|^{2^*_\mu})}{|x-y|^{\mu}}dxdy\\
&~~~+\int_{\Omega}\int_{\Omega}\frac{(|u_n(x)|^{2^*_\mu}-|(u_n-u)(x)|^{2^*_\mu})|(v_n-v)(y)|^{2^*_\mu}}{|x-y|^{\mu}}dxdy\\
&~~~+\int_{\Omega}\int_{\Omega}\frac{(|v_n(x)|^{2^*_\mu}-|(v_n-v)(x)|^{2^*_\mu}) |(u_n-u)(y)|^{2^*_\mu}}{|x-y|^{\mu}}dxdy.
\end{aligned}
\end{eqnarray}
Since $|u_n-u|^{2^*_\mu}\rightharpoonup 0,|v_n-v|^{2^*_\mu}\rightharpoonup 0$ in $L^{\frac{2N}{2N-\mu}}(\Omega)$
as $n\rightarrow\infty$. From $\eqref{4.8}-\eqref{4.12}$, we know that the result holds.
\end{proof}

\begin{lemma}\label{lemma4.7}(Boundedness)
The $(PS)_c$ sequence $\{(u_n,v_n)\}\subset Y(\Omega)$ is bounded.
\end{lemma}
\begin{proof}
From \eqref{2.5} and the definition of $\lambda_{1,s}$,   we have
\begin{eqnarray*}
\begin{aligned}
C+C\|\left(u_n,v_n\right)\|_Y &\geq
J_s\left(u_n,v_n\right)-\frac{1}{2\cdot2^*_\mu}J'_s\left(u_n,v_n)(u_n,v_n\right)\\
&=\left(\frac{1}{2}-\frac{1}{2\cdot2^*_\mu}\right)\|\left(u_n,v_n\right)\|_Y^2\\
&~~-\left(\frac{1}{2}-\frac{1}{2\cdot2^*_\mu}\right)\int_{\R^N}\left(A(u_n,v_n),(u_n,v_n)\right)_{\R^2}dx\\
&\geq \left(\frac{1}{2}-\frac{1}{2\cdot2^*_\mu}\right)(1-\frac{\mu_2}{\lambda_{1,s}})\|(u_n,v_n)\|_Y^2.
\end{aligned}
\end{eqnarray*}
Since $\mu_2 <\lambda_{1,s}$, the assertion follows.
\end{proof}

\begin{lemma}\label{lemma4.8}
If $\{(u_n,v_n)\}\subset Y(\Omega)$ be a $(PS)_c$ sequence for the functional $J_s$ with
\begin{eqnarray*}
\begin{aligned}
c<\frac{N+2s-\mu}{2N-\mu}(S_s^H)^{\frac{2N-\mu}{N+2s-\mu}},
\end{aligned}
\end{eqnarray*}
then $\{(u_n,v_n)\}$ has a convergent subsequence.
\end{lemma}
\begin{proof}
Let $(u_0,v_0)$ be the weak limit of $\{(u_n,v_n)\}$ and define $w_n:=u_n-u_0$, $z_n:=v_n-v_0$, then we know
$w_n\rightharpoonup 0, z_n\rightharpoonup 0$ in $H^s(\R^N)$ and $w_n\rightarrow 0$ a.e. in $\R^N, z_n\rightarrow 0$ a.e. in $\R^N$. Moreover, by \cite[Lemma 5]{PeSqYa2} and Proposition \ref{proposition4.6}, we know
\begin{eqnarray*}
\begin{aligned}
&\|u_n\|_X^2=\|w_n\|_X^2+\|u_0\|_X^2+o_n(1),\\
&\|v_n\|_X^2=\|z_n\|_X^2+\|v_0\|_X^2+o_n(1),
\end{aligned}
\end{eqnarray*}
and
\begin{eqnarray*}
\begin{aligned}
\int_{\Omega}\int_{\Omega}\frac{|u^+_n(x)|^{2^*_\mu}|v^+_n(y)|^{2^*_\mu}}{|x-y|^{\mu}}dxdy=&\int_{\Omega}\int_{\Omega}\frac{|w^+_n(x)|^{2^*_\mu}|z^+_n(y)|^{2^*_\mu}}{|x-y|^{\mu}}dxdy\\
&+ \int_{\Omega}\int_{\Omega}\frac{|u^+_0(x)|^{2^*_\mu}|v^+_0(y)|^{2^*_\mu}}{|x-y|^{\mu}}dxdy+o_n(1).
\end{aligned}
\end{eqnarray*}
Consequently, we have
\begin{eqnarray*}
\begin{aligned}
c\leftarrow
J_s(u_n,v_n)&=\frac{1}{2}\int_{\R^{2N}}\frac{|u(x)-u(y)|^2+|v(x)-v(y)|^2}{|x-y|^{N+2s}}dxdy\\
&~~-\frac{1}{2}\int_{\R^N}(A(u_n,v_n),(u_n,v_n))_{\R^2}dx-\frac{1}{2^*_\mu}\int_{\Omega}\int_{\Omega}\frac{|u^+_n(x)|^{2^*_\mu}|v^+_n(y)|^{2^*_\mu}}{|x-y|^{\mu}}dxdy\\
&\geq\frac{1}{2}\Big (\int_{\R^{2N}}\frac{|w_n(x)-w_n(y)|^2}{|x-y|^{N+2s}}dxdy+\int_{\R^{2N}}\frac{|u_0(x)-u_0(y)|^2}{|x-y|^{N+2s}}dxdy\\
&~~+\int_{\R^{2N}}\frac{|z_n(x)-z_n(y)|^2}{|x-y|^{N+2s}}dxdy+\int_{\R^{2N}}\frac{|v_0(x)-v_0(y)|^2}{|x-y|^{N+2s}}dxdy\Big)\\
&~~-\frac{\mu_2}{2}\Big(\int_{\R^N}|w_n|^2dx
+\int_{\R^N}|z_n|^2dx+\int_{\R^N}|u_0|^2dx+\int_{\R^N}|v_0|^2dx\Big)\\
&~~-\frac{1}{2^*_\mu}\Big(\int_{\Omega}\int_{\Omega}\frac{|w^+_n(x)|^{2^*_\mu}|z^+_n(y)|^{2^*_\mu}}{|x-y|^{\mu}}dxdy
+ \int_{\Omega}\int_{\Omega}\frac{|u^+_0(x)|^{2^*_\mu}|v^+_0(y)|^{2^*_\mu}}{|x-y|^{\mu}}dxdy\Big)\\
&~~+o_n(1),\\
\end{aligned}
\end{eqnarray*}
so
\begin{eqnarray}\label{4.14}
\begin{aligned}
c &\geq J_s(u_0,v_0)+\frac{1}{2}\Big(\int_{\R^{2N}}\frac{|w_n(x)-w_n(y)|^2}{|x-y|^{N+2s}}dxdy\\
&~~+\int_{\R^{2N}}\frac{|z_n(x)-z_n(y)|^2}{|x-y|^{N+2s}}dxdy\Big)
 -\frac{\mu_2}{2}\Big(\int_{\R^N}|w_n|^2dx+\int_{\R^N}|z_n|^2dx\Big)\\
&~~-\frac{1}{2^*_\mu}\int_{\Omega}\int_{\Omega}\frac{|w^+_n(x)|^{2^*_\mu}|z^+_n(y)|^{2^*_\mu}}{|x-y|^{\mu}}dxdy+o_n(1).\\
\end{aligned}
\end{eqnarray}
From the boundedness of Palais-Smale sequences (see Lemma \ref{lemma4.7}) and compact embedding theorems, passing to a
subsequence if necessary, there exists $(u_0,v_0)\in Y(\Omega)$, such that $(u_n,v_n)\rightharpoonup (u_0,v_0)$
weakly in $Y(\Omega)$ as $n\rightarrow\infty$, $(u_n,v_n)\rightarrow (u_0,v_0)$ a.e. in $\Omega$ and strongly in
$L^r(\Omega)$ for $1\leq r<2^*_s$.
Since $|u_n^+|^{2^*_\mu}\rightharpoonup|u_0^+|^{2^*_\mu} ,|v_n^+|^{2^*_\mu}\rightharpoonup|v_0^+|^{2^*_\mu} \text{in }
L^{\frac{2N}{2N-\mu}}(\Omega) \text{ as  } n\rightarrow \infty$, by the Hardy-Littlewood-Sobolev inequality, the
Riesz potential defines a linear continuous map from $L^{\frac{2N}{2N-\mu}}(\Omega)$ to
$L^{\frac{2N}{\mu}}(\Omega)$, hence
\begin{equation*}
\begin{split}
&\int_{\Omega}\frac{|u_n^+(y)|^{2^*_\mu}}{|x-y|^{\mu}}dy\rightharpoonup \int_{\Omega}\frac{|u_0^+(y)|^{2^*_\mu}}{|x-y|^{\mu}}dy \text{ in } L^{\frac{2N}{\mu}}(\Omega),\\
&\int_{\Omega}\frac{|v_n^+(y)|^{2^*_\mu}}{|x-y|^{\mu}}dy\rightharpoonup \int_{\Omega}\frac{|v_0^+(y)|^{2^*_\mu}}{|x-y|^{\mu}}dy \text{ in } L^{\frac{2N}{\mu}}(\Omega).\\
\end{split}
\end{equation*}
as $n\rightarrow \infty $. Combining these with the fact that
\begin{equation*}
\begin{split}
&|u_n^+|^{2^*_\mu-1}\rightharpoonup|u_0^+|^{2^*_\mu-1} \text{  in }  L^{\frac{2N}{N+2s-\mu}}(\Omega),\\
&|v_n^+|^{2^*_\mu-1}\rightharpoonup|v_0^+|^{2^*_\mu-1} \text{  in } L^{\frac{2N}{N+2s-\mu}}(\Omega) .\\
\end{split}
\end{equation*}
as $n\rightarrow \infty $. we obtain
\begin{equation}\label{*}
\begin{split}
\int_{\Omega}\frac{|u_n^+(y)|^{2^*_\mu}|v_n^+(x)|^{2^*_\mu-1}}{|x-y|^{\mu}}dy\rightharpoonup \int_{\Omega}\frac{|u_0^+(y)|^{2^*_\mu}|v_0^+(x)|^{2^*_\mu-1}}{|x-y|^{\mu}}dy \text{  in } L^{\frac{2N}{N+2s}}(\Omega),\\
\int_{\Omega}\frac{|v_n^+(y)|^{2^*_\mu}|u_n^+(x)|^{2^*_\mu-1}}{|x-y|^{\mu}}dy\rightharpoonup \int_{\Omega}\frac{|v_0^+(y)|^{2^*_\mu}|u_0^+(x)|^{2^*_\mu-1}}{|x-y|^{\mu}}dy  \text{  in } L^{\frac{2N}{N+2s}}(\Omega), \\
\end{split}
\end{equation}
as  $n\rightarrow \infty $. Since, for any $\varphi,\psi\subset X(\Omega)$,
\begin{equation*}
\begin{split}
0\leftarrow&J'_s(u_n,v_n)(\varphi,\psi)\\=&\int_{\R^{2N}}\frac{(u_n(x)-u_n(y))(\varphi(x)-\varphi(y))+(v_n(x)-v_n(y))(\psi(x)-\psi(y))}{|x-y|^{N+2s}}dxdy\\
&-\int_{\Omega}(A(u_n,v_n),(\varphi,\psi))_{\R^2}dx-\int_{\Omega}\int_{\Omega}\frac{|v^+_n(y)|^{2^*_\mu}|u_n^+(x)|^{2^*_\mu-1}\varphi(x)}{|x-y|^\mu}dxdy
\\&-\int_{\Omega}\int_{\Omega}\frac{|u^+_n(y)|^{2^*_\mu}|v_n^+(x)|^{2^*_\mu-1}\psi(x)}{|x-y|^\mu}dxdy
\end{split}
\end{equation*}
Passing to the limit as $n\rightarrow\infty$, we obtain
\begin{equation}\label{4.16}
\begin{split}
&\int_{\R^{2N}}\frac{(u_0(x)-u_0(y))(\varphi(x)-\varphi(y))+(v_0(x)-v_0(y))(\psi(x)-\psi(y))}{|x-y|^{N+2s}}dxdy\\
&-\int_{\Omega}(A(u_0,v_0),(\varphi,\psi))_{\R^2}dx-\int_{\Omega}\int_{\Omega}\frac{|v^+_0(y)|^{2^*_\mu}|u_0^+(x)|^{2^*_\mu-1}\varphi(x)}{|x-y|^\mu}dxdy\\&-\int_{\Omega}\int_{\Omega}\frac{|u^+_0(y)|^{2^*_\mu}|v_0^+(x)|^{2^*_\mu-1}\psi(x)}{|x-y|^\mu}dxdy=0,
\end{split}
\end{equation}
which means that $(u_0,v_0)$ is a  weak solution of the problem \eqref{1.1}.\\
Taking $\varphi=u_0,\psi=v_0$ as a test function in equation \eqref{4.16}, we have
\begin{eqnarray*}
\begin{aligned}
&\int_{\R^{2N}}\frac{|u_0(x)-u_0(y)|^2+|v_0(x)-v_0(y)|^2}{|x-y|^{N+2s}}dxdy\\
&=\int_{\Omega}(A(u_0,v_0),(u_0,v_0))_{\R^2}dx+2\int_{\Omega}\int_{\Omega}\frac{|u^+_0(x)|^{2^*_\mu}|v^+_0(y)|^{2^*_\mu}}{|x-y|^{\mu}}dxdy,
\end{aligned}
\end{eqnarray*}
so for $0<\mu<N$,
\begin{eqnarray}\label{4.17}
\begin{aligned}
J_s(u_0,v_0)=\frac{N+2s-\mu}{2N-\mu}\int_{\Omega}\int_{\Omega}\frac{|u^+_0(x)|^{2^*_\mu}|v^+_0(y)|^{2^*_\mu}}{|x-y|^{\mu}}dxdy\geq0.
\end{aligned}
\end{eqnarray}
Using \eqref{4.14}, \eqref{4.17} and $\int_{\R^N}|w_n|^2dx\rightarrow0,\int_{\R^N}|z_n|^2dx\rightarrow0$, as
$n\rightarrow\infty$, we get
\begin{eqnarray}\label{4.18}
\begin{aligned}
c &\geq \frac{1}{2}\left(\int_{\R^{2N}}\frac{|w_n(x)-w_n(y)|^2}{|x-y|^{N+2s}}dxdy+\int_{\R^{2N}}\frac{|z_n(x)-z_n(y)|^2}{|x-y|^{N+2s}}dxdy\right)\\
&~~-\frac{1}{2^*_\mu}\int_{\Omega}\int_{\Omega}\frac{|w^+_n(x)|^{2^*_\mu}|z^+_n(y)|^{2^*_\mu}}{|x-y|^{\mu}}dxdy+o_n(1).\\
\end{aligned}
\end{eqnarray}
Since $(u_0,v_0)$ is a weak solution of \eqref{1.1}, $(u_0,v_0)$ must be critical point of $J_s$ which gives  \\
$\langle J'_s(u_0,v_0),(u_0,v_0)\rangle=0$, hence
\begin{eqnarray}\label{4.19}
\begin{aligned}
o_n(1)&=\langle J'_s(u_n,v_n),(u_n,v_n)\rangle\\
&=\langle
J'_s(u_0,v_0),(u_0,v_0)\rangle+ \int_{\R^{2N}}\frac{|w_n(x)-w_n(y)|^2}{|x-y|^{N+2s}}dxdy\\
&~~~+\int_{\R^{2N}}\frac{|z_n(x)-z_n(y)|^2}{|x-y|^{N+2s}}dxdy -2\int_{\Omega}\int_{\Omega}\frac{|w^+_n(x)|^{2^*_\mu}|z^+_n(y)|^{2^*_\mu}}{|x-y|^{\mu}}dxdy+o_n(1)\\
&=\int_{\R^{2N}}\frac{|w_n(x)-w_n(y)|^2}{|x-y|^{N+2s}}dxdy+\int_{\R^{2N}}\frac{|z_n(x)-z_n(y)|^2}{|x-y|^{N+2s}}dxdy\\
&~~~-2\int_{\Omega}\int_{\Omega}\frac{|w^+_n(x)|^{2^*_\mu}|z^+_n(y)|^{2^*_\mu}}{|x-y|^{\mu}}dxdy+o_n(1).
\end{aligned}
\end{eqnarray}
From \eqref{4.19}, we know there exists a nonnegative constant $m$ such that
\begin{eqnarray*}
\begin{aligned}
\int_{\R^{2N}}\frac{|w_n(x)-w_n(y)|^2}{|x-y|^{N+2s}}dxdy+\int_{\R^{2N}}\frac{|z_n(x)-z_n(y)|^2}{|x-y|^{N+2s}}dxdy\rightarrow m,
\end{aligned}
\end{eqnarray*}
and
\begin{eqnarray*}
\begin{aligned}
\int_{\Omega}\int_{\Omega}\frac{|w^+_n(x)|^{2^*_\mu}|z^+_n(y)|^{2^*_\mu}}{|x-y|^{\mu}}dxdy\rightarrow\frac{m}{2},
\end{aligned}
\end{eqnarray*}
as $n\rightarrow\infty$. Thus from \eqref{4.18}, we obtain
\begin{eqnarray}\label{4.20}
\begin{aligned}
c\geq\frac{N+2s-\mu}{4N-2\mu}m,
\end{aligned}
\end{eqnarray}
By the definition of the best constant $\widetilde{S}_s^H$, we have
\begin{eqnarray*}
\begin{aligned}
\int_{\R^{2N}}\frac{|u(x)-u(y)|^2+|v(x)-v(y)|^2}{|x-y|^{N+2s}}dxdy\geq
\widetilde{S}_s^H\Big(\int_{\Omega}\int_{\Omega}\frac{|u(x)|^{2^*_\mu}|v(y)|^{2^*_\mu}}{|x-y|^{\mu}}dxdy\Big)^{\frac{N-2s}{2N-\mu}},
\end{aligned}
\end{eqnarray*}
which yields $m\geq \widetilde{S}_s^H(\frac{m}{2})^{\frac{N-2s}{2N-\mu}}$. Thus  we have either $m=0$ or\\
$m\geq \frac{1}{2^{\frac{N-2s}{N+2s-\mu}}}(\widetilde{S}_s^H)^{\frac{2N-\mu}{N+2s-\mu}}$, If
$m\geq \frac{1}{2^{\frac{N-2s}{N+2s-\mu}}}(\widetilde{S}_s^H)^{\frac{2N-\mu}{N+2s-\mu}}$, by   Proposition \ref{proposition4.1} (iii), then we obtain from \eqref{4.20} that
\begin{eqnarray*}
\begin{aligned}
c\geq\frac{N+2s-\mu}{2N-\mu}(S_s^H)^{\frac{2N-\mu}{N+2s-\mu}},
\end{aligned}
\end{eqnarray*}
which contradicts
with the fact that $c<\frac{N+2s-\mu}{2N-\mu}(S_s^H)^{\frac{2N-\mu}{N+2s-\mu}}$. Thus $m=0$, and
\begin{eqnarray*}
\begin{aligned}
\|(u_n-u_0,v_n-v_0)\|_Y\rightarrow 0,
\end{aligned}
\end{eqnarray*}
as $n\rightarrow\infty$. This completes the proof of Lemma \ref{lemma4.8}.
\end{proof}
\subsection{Mountain Pass geometry}
\begin{lemma}\label{lemma4.9}
Suppose $\mu_2<\lambda_{1,s}$. The functional $J_s$ satisfies
\begin{enumroman}
\item \label{Proposition 4.9.a} There exist $\beta,\rho>0$ such that $J_s(u,v)\geq\beta$, if $\|(u,v)\|_Y=\rho$;
\item \label{Proposition 4.9.b} there exists $(e_1,e_2)\in Y(\Omega)\backslash \{(0,0)\}$ with
    $\|(e_1,e_2)\|_Y>\rho$ such that $J_s(e_1,e_2)\leq 0$.
\end{enumroman}
\end{lemma}
\begin{proof}
(i) From  the definition of $\widetilde{S}_s^H$, we get
\begin{equation}\label{4.21}
\begin{split}
\int_{\Omega}\int_{\Omega}\frac{|u^+(x)|^{2^*_\mu}|v^+(y)|^{2^*_\mu}}{|x-y|^{\mu}}dxdy
&\leq \frac{1}{(\widetilde{S}^H_s)^{2^*_\mu}}\|(u,v)\|_Y^{2\cdot 2^*_\mu}.\\
\end{split}
\end{equation}
Combining with \eqref{2.5} and the definition of $\lambda_{1,s}$, we get
\begin{equation*}
\begin{split}
J_s(u,v)\geq\frac{1}{2}(1-\frac{\mu_2}{\lambda_{1,s}})\|(u,v)\|_Y^2-\frac{1}{2^*_\mu(\widetilde{S}^H_s)^{2^*_\mu}}\|(u,v)\|_Y^{2\cdot 2^*_\mu}.
\end{split}
\end{equation*}
Since $2<2\cdot2^*_\mu$ and thus, some $\beta, \rho>0$ can be chosen such that $J_s(u,v)\geq\beta$ for $\|(u,v)\|_Y=\rho$. \\
(ii) Choose$(\widetilde{u_0},\widetilde{v_0})\in Y(\Omega)\backslash\{(0,0)\}$ with $\widetilde{u_0}\geq
0,\widetilde{v_0}\geq0$ a.e. and $\widetilde{u_0}\widetilde{v_0}\neq 0$. Then
\begin{equation*}
\begin{split}
J_s(t\widetilde{u_0},t\widetilde{v_0})&=\frac{t^2}{2}\int_{\R^{2N}}\frac{|\widetilde{u_0}(x)-\widetilde{u_0}(y)|^2+|\widetilde{v_0}(x)-\widetilde{v_0}(y)|^2}{|x-y|^{N+2s}}dxdy\\
&~~~-\frac{t^2}{2}\int_{\R^N}(A(\widetilde{u_0},\widetilde{v_0}),(\widetilde{u_0},\widetilde{v_0}))dx-\frac{t^{2\cdot2^*_\mu}}{2^*_\mu}\int_{\Omega}\int_{\Omega}\frac{|\widetilde{u_0}(x)|^{2^*_\mu}|\widetilde{v_0}(y)|^{2^*_\mu}}{|x-y|^{\mu}}dxdy,\\
\end{split}
\end{equation*}
by choosing $t>0$  sufficiently large, the assertion follows. This concludes the proof.
\end{proof}




\begin{lemma}\label{lemma4.10}
If $(u,v)\subset Y(\Omega)$ is a critical point of $J_s$, then $(u^-, v^-)=(0,0)$.
\end{lemma}
\begin{proof}
By choosing $\varphi:=u^-\in X(\Omega)$ and $\psi:=v^-\in X(\Omega)$ as test functions in \eqref{4.13} and
using the elementary inequality
\begin{equation*}
(w_1-w_2)(w_1^{-}-w_2^{-})\geq (w_1^--w_2^-)^2, \text{ for all } w_1,w_2\in \R,
\end{equation*}
we obtain
\begin{equation*}
\begin{split}
& \int_{\R^{2N}}\frac{(u(x)-u(y))(u^-(x)-u^-(y))+(v(x)-v(y))(v^-(x)-v^-(y))}{|x-y|^{N+2s}}dxdy\\
& \geq \int_{\R^{2N}}\frac{(u^-(x)-u^-(y))^2+(v^-(x)-v^-(y))^2}{|x-y|^{N+2s}}dxdy.\\
\end{split}
\end{equation*}
Now, note that, since $b\geq 0$ and $w^-\leq 0$ and $w^+\geq 0$, it holds
\begin{equation*}
\int_{\R^N}(A(u,v),(u^-,v^-))_{\R^2}dx\leq \int_{\R^N}(A(u^-,v^-),(u^-,v^-))_{\R^2}dx.
\end{equation*}
In fact, it follows
\begin{equation*}
\begin{split}
(A(u,v),(u^-,v^-))_{\R^2}&=(A(u^-,v^-),(u^-,v^-))_{\R^2}+b(v^+u^-+u^+v^-),\\
&\leq (A(u^-,v^-),(u^-,v^-))_{\R^2}.
\end{split}
\end{equation*}
In turn, from the formula for $J'_s(u,v)(u^-,v^-)$, it follows that
\begin{equation*}
\begin{split}
J'_s(u,v)(u^-,v^-) \geq & \int_{\R^{2N}}\frac{(u^-(x)-u^-(y))^2+(v^-(x)-v^-(y))^2}{|x-y|^{N+2s}}dxdy\\
&-\int_{\Omega}(A(u^-,v^-),(u^-,v^-))_{\R^2} dx\geq I(u^-)+I(v^-),
\end{split}
\end{equation*}
where we have set
\begin{eqnarray*}
\begin{aligned}
I(w):=\int_{\R^{2N}}\frac{(w(x)-w(y))^2}{|x-y|^{N+2s}}dxdy-\mu_2\int_{\Omega}|w|^2dx=\|w\|^2_X-\mu_2\|w\|^2_{L^2(\Omega)}.
\end{aligned}
\end{eqnarray*}
On the other hand, by definition of $\lambda_{1,s}$, we have
\begin{equation*}
\begin{split}
I(w)\geq ( 1-\frac{\mu_2}{\lambda_{1,s}})\|w\|^2_X,
\end{split}
\end{equation*}
which finally yields the inequality
\begin{equation*}
\begin{split}
J'_s(u,v)(u^-,v^-)\geq (1-\frac{\mu_2}{\lambda_{1,s}})(\|u^-\|^2_X+\|v^-\|^2_X).
\end{split}
\end{equation*}
Since $\{(u,v)\}\subset Y(\Omega)$ is a critical point of $J_s$, we get $J'_s(u,v)(u^-,v^-)=0$, from which that assertion immediately follows.
\end{proof}

\noindent\textbf{Proof of Theorem 1.3}~By Lemma \ref{lemma4.9} and the Mountain Pass Theorem, there exists a sequence
$\{(u_n,v_n)\}\subset Y(\Omega)$, so called $(PS)_c$-Palais Smale sequence at level $c$, such that
\begin{equation}\label{4.22}
\begin{split}
J_s(u_n,v_n)\rightarrow c, J'_s(u_n,v_n)\rightarrow 0,
\end{split}
\end{equation}
where $c$ is given by
\begin{equation*}
\begin{split}
c=\inf_{\gamma\in \Gamma}\max_{t\in [0,1]}J_s(\gamma(t)),
\end{split}
\end{equation*}
with
\begin{equation*}
\begin{split}
\Gamma=\{\gamma \in C([0,1],Y(\Omega)):\gamma(0)=(0,0) \text{ and } J_s(\gamma(1))\leq 0\}.
\end{split}
\end{equation*}
From \eqref{2.5}, we know that there exists $u_\epsilon$ such that
\begin{eqnarray*}
\begin{aligned}
J_s(tu_\epsilon,tu_\epsilon)&\leq
t^2\|u_\epsilon\|^2-\mu_1t^2\|u_\epsilon\|^2_{L^2}-\frac{t^{2\cdot2^*_\mu}}{2^*_\mu}\int_{\Omega}\int_{\Omega}\frac{|u_\epsilon(x)|^{2^*_\mu}|u_\epsilon(y)|^{2^*_\mu}}{|x-y|^{\mu}}dxdy\\
&:=f(t).
\end{aligned}
\end{eqnarray*}
It is easy  to verify that $f(t)$ attains its maximum at
$t_*=\left[\frac{\|u_\epsilon\|^2-\mu_1\|u_\epsilon\|^2_{L^2}}{\int_{\Omega}\int_{\Omega}\frac{|u_\epsilon(x)|^{2^*_\mu}|u_\epsilon(y)|^{2^*_\mu}}{|x-y|^{\mu}}dxdy}\right]^{\frac{1}{2\cdot2^*_\mu-2}}$.
By the definition of $S_{s,\lambda}(v)$ and Lemma \ref{lemma4.5}, we have
\begin{eqnarray*}
\begin{aligned}
c&\leq \sup_{t\geq0}J_s(tu_\epsilon,tu_\epsilon)\\
&\leq f(t_*)=\frac{N+2s-\mu}{2N-\mu}\left[\frac{(\|u_\epsilon\|^2-\mu_1\|u_\epsilon\|^2_{L^2})}{(\int_{\Omega}\int_{\Omega}\frac{|u_\epsilon(x)|^{2^*_\mu}|u_\epsilon(y)|^{2^*_\mu}}{|x-y|^{\mu}}dxdy)^{\frac{1}{2^*_\mu}}}\right]^{\frac{2N-\mu}{N+2s-\mu}}\\
&=\frac{N+2s-\mu}{2N-\mu}(S_{s,\mu_1}(u_\epsilon))^{\frac{2N-\mu}{N+2s-\mu}}<\frac{N+2s-\mu}{2N-\mu}(S_s^H)^{\frac{2N-\mu}{N+2s-\mu}}.
\end{aligned}
\end{eqnarray*}
If one of the following conditions holds,
\begin{enumroman}
\item \label{Theorem 1.2.(1)} $N\geq 4s$ and $\mu_1>0$, or
\item \label{Theorem 1.2.(2)}  $2s<N<4s$ and $\mu_1$ is large enough.
\end{enumroman}
Therefore, combining with Lemma \ref{lemma4.8} and Lemma \ref{lemma4.10},  we get that problem \eqref{1.1} has a nonnegative solution with critical value $c\in (0, \frac{N+2s-\mu}{2N-\mu}(S_s^H)^{\frac{2N-\mu}{N+2s-\mu}})$.


\section{$ Case~3: \xi_1,\xi_2>0,$ $p=q=2^*_\mu$}
In this case ,  we have the  function $J_s:Y(\Omega)\rightarrow \R$ by setting
\begin{eqnarray*}
\begin{aligned}
J_s(U)\equiv J_s(u,v) = &\frac{1}{2}\left(\int_{\R^{2N}}\frac{|u(x)-u(y)^2+|v(x)-v(y)|^2}{|x-y|^{N+2s}}dxdy\right) \\
&-\frac{1}{2}\int_{\R^N}(A(u,v),(u,v))_{\R^2}dx-\frac{1}{2^*_\mu}\Big(\int_{\Omega}\int_{\Omega}\frac{|u(x)|^{2^*_\mu}|v(y)|^{2^*_\mu}}{|x-y|^\mu}dxdy\\
&+\xi_1 \int_{\Omega}\int_{\Omega}\frac{|u(x)|^{2^*_\mu}|u(y)|^{2^*_\mu}}{|x-y|^\mu}dxdy+\xi_2\int_{\Omega}\int_{\Omega}\frac{|v(x)|^{2^*_\mu}|v(y)|^{2^*_\mu}}{|x-y|^\mu}dxdy\Big),
\end{aligned}
\end{eqnarray*}
whose Fr\'{e}chet derivative is given by
\begin{eqnarray*}
\begin{aligned}
J'_s(u,v)(\varphi,\psi)=&\int_{\R^{2N}}\frac{(u(x)-u(y))(\varphi(x)-\varphi(y))+(v(x)-v(y))(\psi(x)-\psi(y))}{|x-y|^{N+2s}}dxdy\\
&-\int_{\Omega}(A(u,v),(\varphi,\psi))_{\R^2}dx-\int_{\Omega}\int_{\Omega}\frac{|u(x)|^{2^*_\mu-2}u(x)|v(y)|^{2^*_\mu}}{|x-y|^\mu}\varphi dxdy
\\&-\int_{\Omega}\int_{\Omega}\frac{|u(x)|^{2^*_\mu}|v(y)|^{2^*_\mu-2}v(y)}{|x-y|^\mu}\psi dxdy-2\xi_1\int_{\Omega}\int_{\Omega}\frac{|u(x)|^{2^*_\mu-2}u(x)|u(y)|^{2^*_\mu}}{|x-y|^\mu}\varphi dxdy\\
&-2\xi_2\int_{\Omega}\int_{\Omega}\frac{|v(x)|^{2^*_\mu}|v(y)|^{2^*_\mu-2}v(y)}{|x-y|^\mu}\psi dxdy,\\
\end{aligned}
\end{eqnarray*}
for every $(\varphi,\psi)\in Y(\Omega)$.
Meanwhile,
\begin{equation}\label{5.1}
F(u,v)=\frac{1}{2^*_\mu}\left[\int_{\Omega}\frac{|v(y)|^{2^*_\mu}}{|x-y|^\mu}dy|u|^{2^*_\mu}+\xi_1 \int_{\Omega}\frac{|u(y)|^{2^*_\mu}}{|x-y|^\mu}dy|u|^{2^*_\mu}+\xi_2\int_{\Omega}\frac{|v(y)|^{2^*_\mu}}{|x-y|^\mu}dy|v|^{2^*_\mu}\right].
\end{equation}
\subsection{Minimizers}
For notational convenience, if $(u,v)\in Y(\Omega)$ we set
\begin{equation*}
B(u,v):=\int_{\Omega}\int_{\Omega}\frac{|u(x)|^{2^*_\mu}|v(y)|^{2^*_\mu}}{|x-y|^\mu}dxdy.
\end{equation*}
and let
\begin{eqnarray}\label{5.2}
\begin{aligned}
\widetilde{S}^H_\xi=\inf_{(u,v)\in
Y(\Omega)\backslash\{(0,0)\}}\frac{\int_{\R^{2N}}\frac{|u(x)-u(y)|^2+|v(x)-v(y)|^2}{|x-y|^{N+2s}}dxdy}{\left(B(u,v)
+\xi_1B(u,u) +\xi_2B(v,v)\right)^{\frac{1}{2^*_\mu}}}.
\end{aligned}
\end{eqnarray}

\begin{Remark}\label{remark5.1}
Let  $T(u,v):=|u|^{2^*_\mu}|v|^{2^*_\mu}+\xi_1|u|^{2\cdot2^*_\mu}+\xi_2|v|^{2\cdot2^*_\mu}$.
It is clear that  $T(u,v)^{\frac{1}{2^*_\mu}}$ is 2-homogeneous, i.e.
\begin{eqnarray*}
\begin{aligned}
T(\varpi U )=\varpi^{2\cdot 2^*_\mu}T( U ), \forall U\in \R^2, \forall \varpi\geq 0.
\end{aligned}
\end{eqnarray*}
there exists a constant
$M > 0$ satisfying
\begin{eqnarray}\label{5.3}
\begin{aligned}
T(u,v)^{\frac{1}{2^*_\mu}}\leq M(|u|^2+|v|^2), \text{ for all } u,v \in \R.
\end{aligned}
\end{eqnarray}
where $M$ is the maximum of the function $T(u,v)^{\frac{1}{2^*_\mu}}$ attained in some $(s_0,t_0)$ of the
compact set $\{(s,t):s,t \in \R, |s|^2+|t|^2=2\}$.
Let $m=M^{-1}$, so we have that
\begin{eqnarray}\label{5.4}
\begin{aligned}
T(s_0,t_0)^{\frac{1}{2^*_\mu}}=m^{-1}(s_0^2+t_0^2).
\end{aligned}
\end{eqnarray}
\end{Remark}

The following result shows the relation between $S^H_s$ and $\widetilde{S}^H_\xi$. The proof is similar to  \cite[Lemma 2.3]{MRspectrum}.
\begin{lemma}\label{lemma5.2}
Let $\Omega$ be a smooth bounded domain, then
\begin{eqnarray}\label{5.5}
\begin{aligned}
\widetilde{S}^H_\xi=mS^H_s.
\end{aligned}
\end{eqnarray}
Moreover, if $g_0$ realizes $S^H_s$ then $(s_0g_0,t_0g_0)$  realizes $\widetilde{S}^H_\xi$, for some
$s_0,t_0>0$.
\end{lemma}
\begin{proof}
Let $\{g_n\}\subset X(\Omega)\backslash \{0\}$ be a minimizing sequence for $S^H_s $ and  consider the sequence
$(\widetilde{u_n},\widetilde{v_n})=(s_0g_n,t_0g_n)$. Substituting $(\widetilde{u_n},\widetilde{v_n})$ in
quotient \eqref{5.2}, we get
\begin{eqnarray}\label{5.6}
\begin{aligned}
\frac{(s_0^2+t_0^2)\|g_n\|^2}{(s_0^{2^*_\mu}t_0^{2^*_\mu}+\xi_1s_0^{2\cdot2^*_\mu}+\xi_2t_0^{2\cdot2^*_\mu})^{\frac{1}{2^*_\mu}}B(g_n,g_n)^{\frac{1}{2^*_\mu}}}\geq
\widetilde{S}^H_\xi.
\end{aligned}
\end{eqnarray}
and consequently by \eqref{5.4} follows that
\begin{eqnarray}\label{5.7}
\begin{aligned}
m\frac{\|g_n\|^2}{B(g_n,g_n)^{\frac{1}{2^*_\mu}}}\geq \widetilde{S}^H_\xi.
\end{aligned}
\end{eqnarray}
Taking the limit in \eqref{5.7}, we obtain
\begin{eqnarray*}
\begin{aligned}
mS^H_s\geq \widetilde{S}^H_\xi.
\end{aligned}
\end{eqnarray*}
In order to prove the reversed inequality, let $\{(u_n,v_n)\}$ be a minimizing sequence for
$\widetilde{S}^H_\xi$. We set $u_n=r_nv_n$ for $r_n>0$. By Proposition \ref{proposition2.3}, we obtain
\begin{equation}\label{5.8}
\frac{\|(u_n,v_n)\|^2}{\left(B(u_n,v_n) +\xi_1B(u_n,u_n) +\xi_2B(v_n,v_n)\right)^{\frac{1}{2^*_\mu}}}\geq
\frac{(1+\frac{1}{r_n^2})S^H_s}{\left(\frac{1}{r^{2^*_\mu}_n}+\xi_1+\xi_2\frac{1}{r^{2\cdot2^*_\mu}_n}\right)^{\frac{1}{2^*_\mu}}}.
\end{equation}
Now, by inequality \eqref{5.3}, we obtain
\begin{equation}\label{5.9}
m\left(\frac{1}{r^{2^*_\mu}_n}+\xi_1+\xi_2\frac{1}{r^{2\cdot2^*_\mu}_n}\right)^{\frac{1}{2^*_\mu}}\leq
1+\frac{1}{r_n^2}.
\end{equation}
Hence, using the inequalities \eqref{5.8} and \eqref{5.9}, we have
\begin{eqnarray*}
\begin{aligned}
\frac{\|(u_n,v_n)\|^2}{\left(B(u_n,v_n) +\xi_1B(u_n,u_n) +\xi_2B(v_n,v_n)\right)^{\frac{1}{2^*_\mu}}}\geq
mS^H_s.
\end{aligned}
\end{eqnarray*}
Therefore, passing to the limit in the above inequality, we have the desired reversed inequality.
\end{proof}

\subsection{Compactness convergence}

\begin{lemma}\label{lemma5.3}(Boundedness)
The $(PS)_c$  sequence $\{(u_n,v_n)\}\subset Y(\Omega)$ is bounded.
\end{lemma}
\begin{proof}
Let $U_n\in Y(\Omega)$ be a $(PS)_c$ sequence, we have
\begin{equation}\label{10}
\begin{split}
J_s(U_n)-\frac{1}{2}\langle J'_s(U_n),U_n\rangle=(2^*_\mu-1)\int_{\Omega}F(U_n)dx\leq \widetilde{C_1}(1+\|U_n\|_Y).
\end{split}
\end{equation}
for some positive constant $\widetilde{C_1}$. From \eqref{2.5}, we have
\begin{equation}\label{11}
\begin{split}
J_s(U_n)+\frac{1}{2}\langle J'_s(U_n),U_n\rangle&=\|U_n\|_Y^2-\int_{\Omega}(A(u,v),(u,v))_{\R^2}dx-(2^*_\mu+1)\int_{\Omega}F(U_n)dx\\
&\leq\|U_n\|_Y^2-\mu_1\|U_n\|^2_{L^2}-(2^*_\mu+1)\int_{\Omega}F(U_n)dx\\
&\leq \widetilde{C_2}(1+\|U_n\|_Y).
\end{split}
\end{equation}
for some positive constant $\widetilde{C_2}$.
Recalling that $2^*_s>2$, by H$\ddot{o}$lder inequality  and \cite[Lemma 2.2]{MRWangY}, we get
\begin{equation*}
\begin{split}
&\|u_n\|^2_{L^2}\leq |\Omega|^{\frac{2s}{N}}\|u_n\|^2_{L^{2^*_s}}\leq \widetilde{C_3}\left(\int_{\Omega}\int_{\Omega}\frac{|u_n(x)|^{2^*_\mu}|u_n(y)|^{2^*_\mu}}{|x-y|^{\mu}}dxdy\right)^{\frac{1}{2^*_\mu}},\\
&\|v_n\|^2_{L^2}\leq |\Omega|^{\frac{2s}{N}}\|v_n\|^2_{L^{2^*_s}}\leq \widetilde{C_4}\left(\int_{\Omega}\int_{\Omega}\frac{|v_n(x)|^{2^*_\mu}|v_n(y)|^{2^*_\mu}}{|x-y|^{\mu}}dxdy\right)^{\frac{1}{2^*_\mu}}.
\end{split}
\end{equation*}
for some positive constant $\widetilde{C_3}$ and $\widetilde{C_4}$.
Combining with \eqref{10}, we get
\begin{equation}\label{12}
\begin{split}
\|U_n\|^2_{L^2}&\leq \widetilde{C_5}\left(\left(\int_{\Omega}\int_{\Omega}\frac{|u_n(x)|^{2^*_\mu}|u_n(y)|^{2^*_\mu}}{|x-y|^{\mu}}dxdy\right)^{\frac{1}{2^*_\mu}}+\left(\int_{\Omega}\int_{\Omega}\frac{|v_n(x)|^{2^*_\mu}|v_n(y)|^{2^*_\mu}}{|x-y|^{\mu}}dxdy\right)^{\frac{1}{2^*_\mu}}\right)\\
&\leq \widetilde{C_5}\left(\int_{\Omega}F(U_n)dx\right)^{\frac{1}{2^*_\mu}}\leq \widetilde{C_6}(1+\|U_n\|_Y)^{\frac{1}{2^*_\mu}}.
\end{split}
\end{equation}
for some positive constant $\widetilde{C_5}$ and $\widetilde{C_6}$.
Hence, by \eqref{10}-\eqref{12}, we conclude that
\begin{equation*}
\begin{split}
\|U_n\|_Y^2 \leq \widetilde{C_7}(1+\|U_n\|_Y)+\widetilde{C_8}(1+\|U_n\|_Y)^{\frac{1}{2^*_\mu}}.
\end{split}
\end{equation*}
for some positive constant $\widetilde{C_7}$  and $\widetilde{C_8}$ . Therefore, we conclude that the sequence $\{U_n\}$ is bounded.\\
\end{proof}

Since $\{U_n\}$ is bounded in $ Y(\Omega)$, up to a subsequence, still denoted by $U_n$, there exists
$U=(u_0,v_0)\in Y(\Omega)$ such that
\begin{equation}\label{5.10}
\begin{split}
U_n\rightharpoonup U \text{ in }  Y(\Omega),
\end{split}
\end{equation}
\begin{equation}\label{5.11}
\begin{split}
U_n\rightharpoonup U \text{ in }  L^{2^*_s}(\Omega)\times L^{2^*_s}(\Omega),
\end{split}
\end{equation}
\begin{equation}\label{5.12}
\begin{split}
U_n\rightarrow U \text{ a.e in }  \Omega,
\end{split}
\end{equation}
\begin{equation}\label{5.13}
\begin{split}
U_n \rightarrow U \text{ in }  L^{r}(\Omega)\times L^{r}(\Omega), \text{ for all }r\in [1,2^*_s),
\end{split}
\end{equation}

\begin{lemma}\label{proposition5.4}
The following relations hold true:
\begin{enumroman}
\item \label{lemma 5.4.a} $J_s(U)=(2^*_\mu-1)\int_{\Omega} F(U)dx\geq 0$.
\item \label{lemma 5.4.b} $J_s(U_n)=J_s(U)+\frac{1}{2}\|U_n-U\|_Y^2-\int_{\Omega} F(U_n-U)dx+o(1)$.
\item \label{lemma 5.10.c} $\|U_n-U\|_Y^2=2\cdot2^*_\mu \int_{\Omega}F(U_n-U)dx+o(1)$.
\end{enumroman}
\end{lemma}
\begin{proof}
i)
Since $|u_n|^{2^*_\mu}\rightharpoonup|u_0|^{2^*_\mu} ,|v_n|^{2^*_\mu}\rightharpoonup|v_0|^{2^*_\mu} \text{ in }
L^{\frac{2N}{2N-\mu}}(\Omega) \text{ as  } n\rightarrow \infty$, by \eqref{*}, we get
\begin{equation}\label{5.14}
\begin{split}
\nabla F(U_n)\rightharpoonup \nabla F(U) \text{ in } L^{\frac{2N}{N+2s}}(\Omega)\times
L^{\frac{2N}{N+2s}}(\Omega).
\end{split}
\end{equation}
So for any $\Theta \in Y(\Omega)$, $\int_{\Omega}(\nabla F(U_n),\Theta)_{\R^2}dx\rightarrow \int_{\Omega}(\nabla F(U),\Theta)_{\R^2}dx$, we have
\begin{equation}\label{5.15}
\begin{split}
J'_s(U_n)(\Theta) =o(1).
\end{split}
\end{equation}
Passing to the limit in \eqref{5.15} as  $n\rightarrow\infty$,  and
combining with the above convergences, we obtain
\begin{equation}\label{5.16}
\begin{split}
\langle U,\Theta \rangle_Y-\int_{\Omega}(AU,\Theta)_{\R^2}dx-\int_{\Omega}(\nabla F(U),\Theta)dx=0,  ~\forall\Theta
\in Y(\Omega)
\end{split}
\end{equation}
which means  $U$ is a weak solution of \eqref{1.1}.\\
Notice that the nonlinearity $F$ is $2\cdot2^*_\mu$-homogeneous, particularly,  we have
\begin{equation}\label{5.17}
(\nabla F(U),U)_{\R^2}=uF_u(U)+vF_v(U)=2\cdot2^*_\mu F(U),  ~\forall U=(u,v)\in \R^2.
\end{equation}
Combined with $J'_s(U)U=0$,
we reach the conclusion.\\
ii) By Lemma \ref{lemma5.3}, the sequence $U_n$  is bounded in $Y(\Omega)\hookrightarrow L^{2^*_s}(\Omega)\times
L^{2^*_s}(\Omega)$,  hence
$U_n$ is bounded in $L^{2^*_s}(\Omega)\times L^{2^*_s}(\Omega)$. Since $U_n\rightarrow U$  a.e. in $\Omega$,  by
the Br\'{e}zis-Lieb Lemma, we have
\begin{equation}\label{5.18}
\begin{split}
\|U_n\|_Y^2=\|U_n-U\|_Y^2+\|U\|_Y^2+o(1).
\end{split}
\end{equation}
\begin{equation}\label{5.19}
\begin{split}
\|U_n\|_{L^{2^*_s}}^2=\|U_n-U\|_{L^{2^*_s}}^2+\|U\|_{L^{2^*_s}}^2+o(1).
\end{split}
\end{equation}
By  Proposition \ref{proposition4.6}, we get
\begin{equation}\label{5.20}
\begin{split}
\int_{\Omega} F(U_n)dx=\int_{\Omega} F(U)dx+\int_{\Omega} F(U_n-U)dx+o(1), \text{ as }n\rightarrow\infty.
\end{split}
\end{equation}
Therefore, using that $U_n\rightarrow U $ in $L^r(\Omega)\times L^r(\Omega)$, for all $r\in [1,2^*_s)$,  by the
definition of $J_s$, \eqref{5.18}, \eqref{5.19} and \eqref{5.20},  we deduce ii).\\
iii) By \eqref{5.11}, \eqref{5.14} and \eqref{5.17}, we get
\begin{equation}\label{5.21}
\begin{split}
&\int_{\Omega}(\nabla F(U_n)-\nabla F(U), U_n-U)_{\R^2}dx\\
&=\int_{\Omega}(\nabla F(U_n), U_n)_{\R^2}dx-\int_{\Omega}(\nabla F(U), U)_{\R^2}dx+o(1)\\
&=2\cdot 2^*_\mu \int_{\Omega}F(U_n)dx-2\cdot 2^*_\mu \int_{\Omega}F(U)dx+o(1).
\end{split}
\end{equation}
Therefore, using \eqref{5.20}, we get
\begin{equation}\label{5.22}
\begin{split}
\int_{\Omega}(\nabla F(U_n)-\nabla F(U), U_n-U)_{\R^2}dx=2\cdot 2^*_\mu\int_{\Omega}F(U_n-U)dx+o(1).
\end{split}
\end{equation}
On the other hand,
\begin{equation}\label{5.23}
\begin{split}
o(1)&=J'_s(U_n)(U_n-U)=J'_s(U_n)(U_n-U)-J'_s(U)(U_n-U)\\
&=\langle U_n, U_n-U \rangle_Y-\int_{\Omega}(AU_n, U_n-U)_{\R^2}dx-\int_{\Omega}(\nabla F(U_n),
U_n-U)_{\R^2}dx\\
&~~~-\langle U, U_n-U \rangle_Y+\int_{\Omega}(AU, U_n-U)_{\R^2}dx+\int_{\Omega}(\nabla F(U),  U_n-U)_{\R^2}dx\\
&=\langle U_n-U, U_n-U \rangle_Y-\int_{\Omega}(A(U_n-U), U_n-U)_{\R^2}dx\\
&~~~-\int_{\Omega}(\nabla F(U_n)-\nabla F(U),  U_n-U)_{\R^2}dx.
\end{split}
\end{equation}
Hence, from \eqref{5.13} and \eqref{5.22},  it follows that
\begin{equation*}
\begin{split}
\|U_n-U\|_Y^2=2\cdot2^*_\mu\int_{\Omega} F(U_n-U)dx+o(1), \text{ as } n\rightarrow\infty.
\end{split}
\end{equation*}
This concludes the proof.
\end{proof}

\begin{lemma}\label{lemma5.5}
Let $N>2s$, $0<\mu<N$ and $\{U_n\}$ be a $(PS)_c$ sequence of $J_s$ with
\begin{equation}\label{5.24}
\begin{split}
c<\frac{N+2s-\mu}{2N-\mu}\left(\frac{\widetilde{S}^H_\xi}{2}\right)^{\frac{2N-\mu}{N+2s-\mu}}.
\end{split}
\end{equation}
Then, $\{U_n\}$ has a convergent subsequence.
\end{lemma}
\begin{proof}
We  assume that
\begin{equation}\label{5.26}
\begin{split}
\|U_n-U\|_Y^2\rightarrow L, \text{ as } n\rightarrow\infty.
\end{split}
\end{equation}
From Lemma \ref{proposition5.4} (iii),
\begin{equation}\label{5.27}
\begin{split}
2\cdot 2^*_\mu\int_{\Omega}F(U_n-U)dx \rightarrow L, \text{ as } n\rightarrow\infty
\end{split}
\end{equation}
and consequently $L\in [0, \infty)$. By the definition of $\widetilde{S}^H_\xi$,  we have
\begin{equation*}
\begin{split}
L\geq \widetilde{S}^H_\xi\left(\frac{L}{2}\right)^{\frac{1}{ 2^*_\mu}}
\end{split}
\end{equation*}
and consequently, either
\begin{equation*}
\begin{split}
L=0 \text{ or } L\geq \left(\frac{1}{2}\right)^{\frac{N-2s}{N+2s-\mu}}(\widetilde{S}^H_\xi)^{\frac{2N-\mu}{N+2s-\mu}}.
\end{split}
\end{equation*}
If $L\geq \left(\frac{1}{2}\right)^{\frac{N-2s}{N+2s-\mu}}(\widetilde{S}^H_\xi)^{\frac{2N-\mu}{N+2s-\mu}}$, from Lemma \ref{proposition5.4} (iii), it is follows that
\begin{equation*}
\begin{split}
\frac{1}{2}\|U_n-U\|_Y^2-\int_{\Omega} F(U_n-U)dx=\frac{N +2s-\mu}{2(2N-\mu)}\|U_n-U\|_Y^2+o(1).
\end{split}
\end{equation*}
Therefore, using  Lemma \ref{proposition5.4} (ii) and above equality, we see that
\begin{equation}\label{5.25}
\begin{split}
J_s(U)+\frac{N +2s-\mu}{2(2N-\mu)}\|U_n-U\|_Y^2&=J_s(U)+\frac{1}{2}\|U_n-U\|_Y^2-\int_{\Omega} F(U_n-U)dx+o(1)\\
&=J_s(U_n)+o(1)=c+o(1), \text{ as } n\rightarrow\infty.
\end{split}
\end{equation}
So
\begin{equation*}
\begin{split}
c&=J_s(U)+\frac{N +2s-\mu}{2(2N-\mu)}L\geq \frac{N +2s-\mu}{2(2N-\mu)}L\\
&\geq
\frac{N +2s-\mu}{2(2N-\mu)}\left(\frac{1}{2}\right)^{\frac{N-2s}{N+2s-\mu}}(\widetilde{S}^H_\xi)^{\frac{2N-\mu}{N+2s-\mu}}
=\frac{N +2s-\mu}{2N-\mu}\left(\frac{\widetilde{S}^H_\xi}{2}\right)^{\frac{2N-\mu}{N+2s-\mu}},
\end{split}
\end{equation*}
which contradicts \eqref{5.24}. Thus $L=0$ and therefore, by \eqref{5.26}, we have
\begin{equation*}
\begin{split}
\|U_n-U\|_Y^2\rightarrow 0, \text{ as } n\rightarrow\infty
\end{split}
\end{equation*}
and so the assertion of Lemma \ref{lemma5.5}  follows.
\end{proof}

\subsection{Linking geometry}

\begin{lemma}\label{lemma5.6}
If $\mathbb{F}$ is a finite dimensional subspace of $Y(\Omega)$,  then there exists $R>0$  large enough such that
 $J_s(u,v)\leq 0$, for all $(u,v)\in \mathbb{F}$ with $\|(u,v)\|_Y\geq R$ and $uv\neq0$.
\end{lemma}
\begin{proof}
Choose $(\widetilde{u_0},\widetilde{v_0})\in \mathbb{F}$ with $\widetilde{u_0}\widetilde{v_0}\neq0$, then
\begin{eqnarray*}
\begin{aligned}
J_s(t\widetilde{u_0},t\widetilde{v_0})&=\frac{t^2}{2}\int_{\R^{2N}}\frac{|\widetilde{u_0}(x)-\widetilde{u_0}(y)|^2+|\widetilde{v_0}(x)-\widetilde{v_0}(y)|^2}{|x-y|^{N+2s}}dxdy\\
&~~~-\frac{t^2}{2}\int_{\R^N}(A(\widetilde{u_0},\widetilde{v_0}),(\widetilde{u_0},\widetilde{v_0}))dx-\frac{t^{2\cdot2^*_\mu}}{2^*_\mu}\Big(\int_{\Omega}\int_{\Omega}\frac{|\widetilde{u_0}(x)|
^{2^*_\mu}|\widetilde{v_0}(x)|^{2^*_\mu}}{|x-y|^\mu}dxdy\\
&~~~+\xi_1\int_{\Omega}\int_{\Omega}\frac{|\widetilde{u_0}(x)|^{2^*_\mu}|\widetilde{u_0}(x)|^{2^*_\mu}}{|x-y|^\mu}dxdy+\xi_2\int_{\Omega}\int_{\Omega}\frac{|\widetilde{v_0}(x)|^{2^*_\mu}|\widetilde{v_0}(x)|^{2^*_\mu}}{|x-y|^\mu}dxdy\Big),
\end{aligned}
\end{eqnarray*}
by choosing $t>0$ large enough, the assertion follows. This concludes the proof.
\end{proof}

\begin{lemma}\label{proposition5.7}
\begin{eqnarray}\label{5.28}
\begin{aligned}
\text{  If      } \lambda_{k-1,s}<\mu_1<\lambda_{k,s}\leq\mu_2<\lambda_{k+1,s}, \text{  for some  } k\geq 1,
\end{aligned}
\end{eqnarray}
 Then the functional $J_s$ satisfies:
\begin{enumroman}
\item \label{lemma 5.5.a} There exist $\alpha,\rho>0$ such that $J_s(u,v)\geq\alpha$ for all $(u,v)\in W$ with
    $\|(u,v)\|_Y=\rho$.
\item \label{lemma 5.5.b} If $Q=(V\cap \overline{B}_R(0))\oplus [0,R]e$, where $e\in W\cap \partial B_1(0)$ is a
    fixed vector, then $J_s(u,v)<0$ for
all  $(u,v)\in \partial Q$ and $R>\rho$ large enough.
\end{enumroman}
\end{lemma}
\begin{proof}
Considering the following subspaces $W=Z_k\oplus H$, where
\begin{eqnarray*}
\begin{aligned}
Z_k=\text{span}\{(\varphi_{k,s},0),(0,\varphi_{k,s})\} \text{ and } H=\text{ span
}\{\overline{(\varphi_{k+1,s},0),(0,\varphi_{k+1,s}),\ldots}\},
\end{aligned}
\end{eqnarray*}
we have that if $U\in W$, then $U=U^k+\overline{U}$ with $U^k\in Z_k$ and $\overline{U}\in H$.
Since $\|U\|_Y^2=\|U^k\|_Y^2+\|\overline{U}\|_Y^2$, by \eqref{2.5} and \eqref{4.21} we have
\begin{eqnarray*}
\begin{aligned}
J_s(U)\geq\frac{1} {2}(\|U^k\|_Y^2+\|\overline{U}\|^2_Y)-\frac{\mu_2}{2}(\|U^k\|_{(L^2)^2}^2+\|\overline{U}\|^2_{(L^2)^2})-C(\|U^k\|_Y^2+\|\overline{U}\|^2_Y)^{2^*_\mu},
\end{aligned}
\end{eqnarray*}
where $U=(u,v)$ and $C:=C(\xi_1,\xi_2)>0$ is a constant.
Therefore, using that $U^k\in Z_k\subset W$ and $\overline{U}\in H$, we obtain
\begin{eqnarray*}
\begin{aligned}
\|U^k\|_{(L^2)^2}^2\leq \frac{1}{\lambda_{k,s}}\|U^k\|_Y^2 \text{ and }  \|\overline{U}\|_{(L^2)^2}^2\leq
\frac{1}{\lambda_{k+1,s}}\|\overline{U}\|_Y^2.
\end{aligned}
\end{eqnarray*}
Consequently,
\begin{eqnarray}\label{5.29}
\begin{aligned}
J_s(U)&\geq
\left(\frac{1}{2}\|\overline{U}\|_Y^2-\frac{\mu_2}{2}\|\overline{U}\|_{(L^2)^2}^2\right)+\left(\frac{1}{2}\|U^k\|_Y^2-\frac{\mu_2}{2}\|U^k\|_{(L^2)^2}^2\right)\\
&~~~-C(\|U^k\|_Y^2+\|\overline{U}\|_Y^2)^{2^*_\mu}\\
&\geq\frac{1}{2}\left(1-\frac{\mu_2}{\lambda_{k+1,s}}\right)\|\overline{U}\|_Y^2+\frac{1}{2}\left(1-\frac{\mu_2}{\lambda_{k,s}}\right)\|U^k\|_Y^2-C\|U^k\|_Y^{2\cdot2^*_\mu}\\
&~~~-C\|\overline{U}\|_Y^{2\cdot2^*_\mu}.
\end{aligned}
\end{eqnarray}
Taking $\|U\|_Y=\rho$  small enough, since $\|U\|_Y^2=\|U^k\|_Y^2+\|\overline{U}\|_Y^2$, we get that
$\|U^k\|_Y:=y(\rho)$ and $\|\overline{U}\|_Y:=z(\rho)\equiv z$  are small enough. Now consider the function
\begin{eqnarray*}
\begin{aligned}
\alpha
(z)&=\frac{1}{2}\left(1-\frac{\mu_2}{\lambda_{k+1,s}}\right)z^2+\frac{1}{2}\left(1-\frac{\mu_2}{\lambda_{k,s}}\right)y(\rho)^2-C(y(\rho)^{2\cdot2^*_\mu}+z^{2\cdot2^*_\mu})\\
&=h(z)+\frac{1}{2}\left(1-\frac{\mu_2}{\lambda_{k,s}}\right)y(\rho)^2-Cy(\rho)^{2\cdot2^*_\mu},
\end{aligned}
\end{eqnarray*}
where $h(z)=\frac{1}{2}\left(1-\frac{\mu_2}{\lambda_{k+1,s}}\right)z^2-Cz^{2\cdot2^*_\mu}$.
By \eqref{5.28},  the maximum value of $h(z)$, for  $\rho$ sufficiently small, is given by
\begin{eqnarray*}
\begin{aligned}
\overline{h}:=\frac{N+2s-\mu}{2N-\mu}\Big(\frac{1}{2^*_\mu C}\Big)^{\frac{N-2s}{N+2s-\mu}}\Big(\frac{1}{2}(1-\frac{\mu_2}{\lambda_{k+1,s}})\Big)^{\frac{2N-\mu}{N+2s-\mu}}>0,
\end{aligned}
\end{eqnarray*}
which is independent of $\rho$ and it is assumed at
\begin{eqnarray*}
\begin{aligned}
\overline{z}:=\left(\frac{1}{2\cdot2^*_\mu C}\right)^{\frac{N-2s}{2(N+2s-\mu)}}\Big(1-\frac{\mu_2}{\lambda_{k+1,s}}\Big)^{\frac{N-2s}{2(N+2s-\mu)}}.
\end{aligned}
\end{eqnarray*}
Therefore, it is possible to choose $y(\rho)$ small enough, such that
\begin{equation*}
\alpha (\overline{z})=\overline{h}-cy(\rho)^2-Cy(\rho)^{2\cdot2^*_\mu}\geq \overline{h}-(c+C)y(\rho)^2>0,
\end{equation*}
where $c=\frac{1}{2}\left(\frac{\mu_2}{\lambda_k}-1\right)\geq0$.
Hence, by the estimate \eqref{5.29} and by the above information, for $\|U\|_Y=\rho$  small enough, there exists
$\alpha>0$ such that $J_s(U)\geq\alpha$. This proves the statement (i).\\
To prove  item (ii), we take $U=(u,v)\in V$, where $u=\Sigma^{k-1}_{i=1}u_i\varphi_{i,s}$,
 $v=\Sigma^{k-1}_{i=1}v_i\varphi_{i,s}$. Using  \cite[Proposition 9]{MRVariational}, we get
\begin{eqnarray*}
\begin{aligned}
\int_{\R^N}|u|^2dx=\Sigma^{k-1}_{i=1}u_i^2, \int_{\R^N}|v|^2dx=\Sigma^{k-1}_{i=1}v_i^2,
\end{aligned}
\end{eqnarray*}
also
\begin{eqnarray*}
\begin{aligned}
\int_{\R^N}
|(-\Delta)^{\frac{s}{2}}u|^2dx+\int_{\R^N}
|(-\Delta)^{\frac{s}{2}}v|^2dx&=\Sigma^{k-1}_{i=1}(u_i^2+v_i^2)\|\varphi_{i,s}\|_X^2\\&=\Sigma^{k-1}_{i=1}(u_i^2+v_i^2)\lambda_{i,s}.
\end{aligned}
\end{eqnarray*}
Then, using \eqref{2.5}, we are going to prove that $J_s(U)<0$ on $V$.
Let $U=(u,v)\in V$, since $\lambda_{k-1,s}<\mu_1<\lambda_{k,s}\leq \mu_2<\lambda_{k+1,s}$,  we have that
\begin{eqnarray*}
\begin{aligned}
J_s(u,v)&\leq\frac{1}{2}\Sigma^{k-1}_{i=1}(u_i^2+v_i^2)\lambda_{i,s}-\frac{\mu_1}{2}\Sigma^{k-1}_{i=1}(u_i^2+v_i^2)\\
&~~~-\frac{1}{2^*_\mu}\Big(\int_{\Omega}\int_{\Omega}\frac{|u(x)|
^{2^*_\mu}|v(x)|^{2^*_\mu}}{|x-y|^\mu}dxdy+\xi_1\int_{\Omega}\int_{\Omega}\frac{|u(x)|^{2^*_\mu}|u(x)|^{2^*_\mu}}{|x-y|^\mu}dxdy\\
&~~~+\xi_2\int_{\Omega}\int_{\Omega}\frac{|v(x)|^{2^*_\mu}|v(x)|^{2^*_\mu}}{|x-y|^\mu}dxdy\Big)\\
&\leq\frac{1}{2}\Sigma^{k-1}_{i=1}(u_i^2+v_i^2)(\lambda_{i,s}-\mu_1)<0.
\end{aligned}
\end{eqnarray*}
Now, to end the proof, it is enough to apply Lemma \ref{lemma5.6} to the finite dimensional subspace $V\oplus \text{
span }\{e\}$  containing $Q=(V\bigcap \overline{B}_R(0))\cap [0,R]e$, for some $e\in W\cap \partial B_1(0)$ and
$R>\rho$.
\end{proof}

\begin{Remark}\label{remark5.8}
Notice that, in  Lemma \ref{proposition5.7} we can choose the finite dimensional subspace $\mathbb{F}$ of $Y(\Omega)$ as
\begin{equation*}
\begin{split}
\mathbb{F}_\epsilon=V\oplus\text{ span }\{ e\}=V\oplus \text{ span }\{(\widetilde{z_\epsilon},0)\},
\end{split}
\end{equation*}
where $V=\text{ span
}\{(0,\varphi_{1,s}),(\varphi_{1,s},0),(0,\varphi_{2,s}),(\varphi_{2,s},0),\ldots,(0,\varphi_{k-1,s}),(\varphi_{k-1,s},0)\}$,
$\widetilde{z_\epsilon}=\frac{z_\epsilon}{\|z_\epsilon\|_X}$, with
$z_\epsilon=u_\epsilon-\Sigma^{k-1}_{j=1}(\int_{\Omega}u_\epsilon\varphi_{j,s}dx)\varphi_{j,s}$.
\end{Remark}

\begin{lemma}\label{lemma5.9}
Let $s\in (0,1)$, $N>2s$ and $M_\epsilon:=\max_{u\in G} S_{s,\mu_1}$, where $G:=\{u \in \mathbb{F}_\epsilon:
\int_{\Omega}\int_{\Omega}\frac{|u(x)|^{2^*_\mu}|u(y)|^{2^*_\mu}}{|x-y|^{\mu}}dxdy=1\}$.
Suppose $\lambda_{k-1,s}<\mu_1<\lambda_{k,s}\leq\mu_2<\lambda_{k+1,s}$,  for some $k\in \N$, then
\begin{enumroman}
\item \label{proposition 5.10.a}
$M_\epsilon $ is achieved by $u_M\in \mathbb{F}_\epsilon$ and $u_M$ can be written as follows
\begin{equation*}
\begin{split}
u_M=\nu+tu_\epsilon, \text{ with } \nu\in \text{ span }\{\varphi_{1,s},\varphi_{2,s},\ldots,\varphi_{k-1,s}\} \text{ and } t\geq0;
\end{split}
\end{equation*}

\item \label{proposition 5.11.b} $M_\epsilon<S^H_s$, provided\\
a)~$N\geq4s$ and $\mu_1>0$, or\\
b)~$2s<N< 4s$ and $\mu_1$  is large enough.\\
\end{enumroman}
\end{lemma}
\begin{proof}
(i)  Thanks to the Weierstrass Theorem,  $M_\epsilon$  is achieved at $u_M$. Since $u_M\in \mathbb{F}_\epsilon$  and by the
definition of $\mathbb{F}_\epsilon$, we have that $u_M=\widetilde{\nu}+tz_\epsilon$,
for some $\widetilde{\nu}\in \text{ span }\{\varphi_{1,s},\varphi_{2,s},\ldots,\varphi_{k-1,s}\}$ and $t\geq 0$.  From the definition of $z_\epsilon$ in Remark \ref{remark5.8}, we have that
\begin{equation}\label{5.30}
\begin{split}
u_M=\nu+tu_\epsilon,
\end{split}
\end{equation}
where
\begin{equation*}
\begin{split}
\nu=\widetilde{\nu}-t\Sigma^{k-1}_{i=1}(\int_{\Omega}u_\epsilon\varphi_{i,s}dx)\varphi_{i,s}\in \text{ span }\{\varphi_{1,s},\varphi_{2,s},\ldots,\varphi_{k-1,s}\}.
\end{split}
\end{equation*}
(ii) First let $t=0$, then $u_M=\nu$ and
\begin{equation*}
\begin{split}
M_\epsilon=\|\nu\|^2-\mu_1\int_{\R^N}|\nu|^2dx\leq (\lambda_{k-1,s}-\mu_1)\|\nu\|^2_{L^2(\Omega)}<0<S_s^H.
\end{split}
\end{equation*}
Now, suppose $t>0$,
we find that $\widetilde{\nu}$ and $z_\epsilon$ are orthogonal in $L^2(\Omega)$, then $\|u_M\|^2_{L^2(\Omega)}=\|\widetilde{\nu}\|^2_{L^2(\Omega)}+\|z_\epsilon\|^2_{L^2(\Omega)}$. Since
$\int_{\Omega}\int_{\Omega}\frac{|u(x)|^{2^*_\mu}|u(y)|^{2^*_\mu}}{|x-y|^{\mu}}dxdy=1$,
using  \cite[Lemma 4.7]{MRmultip}, we get a constant $C_0>0$ (independent of $\epsilon$)   such that $\|u_M\|_{L^{2^*_\mu}(\Omega)}\leq C_0 $. Subsequently, using H\"{o}lder inequality, we get a constant $C_1>0$ (also independent of $\epsilon$ ) such that $\|u_M\|_{L^2(\Omega)}^2\leq C_1 $. Therefore, we can find $C_2>0$ such that $\|u_M\|_{L^2(\Omega)}^2$ and $\|\widetilde{\nu}\|^2_{L^2(\Omega)}$ are both uniformly bounded in $\epsilon$. By computations, we get
\begin{equation}\label{5.31}
\begin{split}
\|u_\epsilon\|^{\frac{3N-2\mu+2s}{N-2s}}_{L^{\frac{N(3N-2\mu+2s)}{(2N-\mu)(N-2s)}}(\Omega)}&=\left(\int_{\Omega}|u_\epsilon|^{\frac{N(3N-2\mu+2s)}{(2N-\mu)(N-2s)}}dx\right)^{\frac{2N-\mu}{N}}\\
&\leq \left(\int_{B_{2\delta}}|U_\epsilon|^{\frac{N(3N-2\mu+2s)}{(2N-\mu)(N-2s)}}dx\right)^{\frac{2N-\mu}{N}}\\
&\leq C_3 \epsilon^{\frac{N-2s}{2}}\left(\int_0^{\frac{2\delta}{\epsilon}}\frac{r^{N-1}}{(1+r^2)^{\frac{N(3N-2\mu+2s)}{(2N-\mu)(N-2s)}}}dr\right)^{\frac{2N-\mu}{N}}\leq O(\epsilon^{\frac{N-2s}{2}}),
\end{split}
\end{equation}
where $C_3>0$ is a constant.  Since $\varphi_{1,s},\varphi_{2,s},\ldots,\varphi_{k-1,s}\in L^\infty (\Omega)$, we have $\widetilde{\nu}\in L^\infty (\Omega)$. Using the fact that the map $t\mapsto t^{2\cdot 2^*_\mu}$ is convex, for $t> 0$ and  span $\{\varphi_{1,s},\varphi_{2,s},\ldots,\varphi_{k-1,s}\}$ is a finite dimensional space,  all norms are equivalent, we get
\begin{equation*}
\begin{split}
1&=\int_{\Omega}\int_{\Omega}\frac{|u_M(x)|^{2^*_\mu}|u_M(y)|^{2^*_\mu}}{|x-y|^\mu}dxdy=\int_{\Omega}\int_{\Omega}\frac{|(\nu+tu_\epsilon)(x)|^{2\cdot 2^*_\mu}}{|x-y|^\mu}dxdy\\
&\geq \int_{\Omega}\int_{\Omega}\frac{|tu_\epsilon(x)|^{2\cdot 2^*_\mu}}{|x-y|^\mu}dxdy+2\cdot 2^*_\mu \int_{\Omega}\int_{\Omega}\frac{|tu_\epsilon(x)|^{2\cdot 2^*_\mu-1}|\nu(x)|}{|x-y|^\mu}dxdy\\
&\geq \int_{\Omega}\int_{\Omega}\frac{|tu_\epsilon(x)|^{ 2^*_\mu}|tu_\epsilon(y)|^{ 2^*_\mu}}{|x-y|^\mu}dxdy\\
&~~~-2\cdot 2^*_\mu\|\nu\|_{L^\infty (\Omega)} \int_{\Omega}\int_{\Omega}\frac{|tu_\epsilon(x)|^{\frac{2\cdot 2^*_\mu-1}{2} }|tu_\epsilon(y)|^{\frac{2\cdot 2^*_\mu-1}{2} }}{|x-y|^\mu}dxdy\\
&\geq \int_{\Omega}\int_{\Omega}\frac{|tu_\epsilon(x)|^{ 2^*_\mu}|tu_\epsilon(y)|^{ 2^*_\mu}}{|x-y|^\mu}dxdy-C_4\|\nu\|_{L^2(\Omega)}\|u_\epsilon\|^{\frac{3N-2\mu+2s}{N-2s}}_{L^{\frac{N(3N-2\mu+2s)}{(2N-\mu)(N-2s)}}(\Omega)}.
\end{split}
\end{equation*}
Combining with \eqref{5.31} with  above inequality, we get
\begin{equation}\label{5.32}
\begin{split}
\int_{\Omega}\int_{\Omega}\frac{|tu_\epsilon(x)|^{ 2^*_\mu}|tu_\epsilon(y)|^{ 2^*_\mu}}{|x-y|^\mu}dxdy\leq 1+C_4\|\nu\|_{L^2(\Omega)}O(\epsilon^{\frac{N-2s}{2}}).
\end{split}
\end{equation}
Hence, using the definition of $S_{s,\mu_1}$ and \eqref{5.30}, we get
\begin{equation}\label{5.33}
\begin{split}
M_\epsilon&=\int_{\R^{2N}}\frac{|u_M(x)-u_M(y)|^2}{|x-y|^{N+2s}}dxdy-\mu_1\int_{\R^N}|u_M(x)|^2dx\\
&=\int_{\R^{2N}}\frac{|(\nu(x)+tu_\epsilon(x))-(\nu(y)+tu_\epsilon(y))|^2}{|x-y|^{N+2s}}dxdy-\mu_1\int_{\R^N}|\nu(x)+tu_\epsilon(x)|^2dx\\
&=\int_{\R^{2N}}\frac{|\nu(x)-\nu(y)|^2}{|x-y|^{N+2s}}dxdy+t^2\int_{\R^{2N}}\frac{|u_\epsilon(x)-u_\epsilon(y)|^2}{|x-y|^{N+2s}}dxdy\\
&~~~+2t\int_{\R^{2N}}\frac{|(\nu(x)-\nu(y))(u_\epsilon(x)-u_\epsilon(y))|}{|x-y|^{N+2s}}dxdy-\mu_1\int_{\R^N}|\nu(x)|^2dx-\mu_1t^2\int_{\R^N}|u_\epsilon(x)|^2dx\\
&~~~-2\mu_1t\int_{\R^N}|u_\epsilon(x)\nu(x)|dx\\
&\leq (\lambda_{k-1,s}-\mu_1)\|\nu\|^2_{L^2(\Omega)}+S_{s,\mu_1}(u_\epsilon)\left(\int_{\Omega}\int_{\Omega}\frac{|tu_\epsilon(x)|^{ 2^*_\mu}|tu_\epsilon(y)|^{ 2^*_\mu}}{|x-u|^\mu}dxdy\right)^\frac{N-2s}{2N-\mu}\\
&~~~+2t\int_{\R^{2N}}\frac{|(\nu(x)-\nu(y))(u_\epsilon(x)-u_\epsilon(y))|}{|x-y|^{N+2s}}dxdy-2\mu_1t\int_{\R^N}|u_\epsilon(x)\nu(x)|dx.
\end{split}
\end{equation}
Now we write $\nu=\Sigma^{k-1}_{i=1}\nu_i\varphi_{i,s}$ for some $\nu_i\in \R$, so that  $\|\nu\|^2_{L^2(\Omega)}=\Sigma^{k-1}_{i=1}\nu_i^2$.  By the H\"{o}lder inequality and the equivalence of the norms in a finite dimensional space,
\begin{equation*}
\begin{split}
|\langle u_\epsilon, \nu\rangle_X|&=\Sigma^{k-1}_{i=1}\lambda_{i,s}\nu_i\int_{\Omega}u_\epsilon(x)\varphi_{i,s}(x)dx\\
&\leq \Sigma^{k-1}_{i=1}\lambda_{i,s}|\nu_i|\|u_\epsilon\|_{L^1(\Omega)}\|\varphi_{i,s}\|_{L^\infty(\Omega)}\\
&\leq\widetilde{k}\lambda_{k,s}\|u_\epsilon\|_{L^1(\Omega)}\|\nu\|_{L^\infty(\Omega)}\\
&\leq \overline{k}\|u_\epsilon\|_{L^1(\Omega)}\|\nu\|_{L^2(\Omega)},
\end{split}
\end{equation*}
for suitable $\widetilde{k}$ and $\overline{k}>0$. More explicitly,
\begin{equation}\label{5.34}
\begin{split}
\left|\int_{\R^{2N}}\frac{(\nu(x)-\nu(y))(u_\epsilon(x)-u_\epsilon(y))}{|x-y|^{N+2s}}dxdy\right|\leq \overline{k}\|u_\epsilon\|_{L^1(\Omega)}\|\nu\|_{L^2(\Omega)}.
\end{split}
\end{equation}
Gathering the results in  \eqref{5.32}, \eqref{5.33} and \eqref{5.34}, using again the H\"{o}lder inequality and Proposition \ref{Proposition4.2} (iv), we get
\begin{equation*}
\begin{split}
M_\epsilon &\leq (\lambda_{k-1,s}-\mu_1)\|\nu\|^2_{L^2(\Omega)}+S_{s,\mu_1}(u_\epsilon)\left(1+C_4\|\nu\|_{L^2(\Omega)}O(\epsilon^{\frac{N-2s}{2}})\right)^{\frac{N-2s}{2N-\mu}}+2t\overline{k}\|u_\epsilon\|_{L^1(\Omega)}\|\nu\|_{L^2(\Omega)}\\
&~~~-2\mu_1t\|u_\epsilon\|_{L^1(\Omega)}\|\nu\|_{L^\infty(\Omega)}\\
&\leq (\lambda_{k-1,s}-\mu_1)\|\nu\|^2_{L^2(\Omega)}+S_{s,\mu_1}(u_\epsilon)\left(1+C_4\|\nu\|_{L^2(\Omega)}O(\epsilon^{\frac{N-2s}{2}})\right)^{\frac{N-2s}{2N-\mu}}+\kappa\|u_\epsilon\|_{L^1(\Omega)}\|\nu\|_{L^2(\Omega)}\\
&\leq (\lambda_{k-1,s}-\mu_1)\|\nu\|^2_{L^2(\Omega)}+S_{s,\mu_1}(u_\epsilon)(1+C_4\|\nu\|_{L^2(\Omega)}O(\epsilon^{\frac{N-2s}{2}}))+O(\epsilon^{\frac{N-2s}{2}})\|\nu\|_{L^2(\Omega)}.
\end{split}
\end{equation*}
Since the parabola $(\lambda_{k-1,s}-\mu_1)\|\nu\|^2_{L^2(\Omega)}+O(\epsilon^{\frac{N-2s}{2}})\|\nu\|_{L^2(\Omega)}$ stays always below its vertex, that is
\begin{equation*}
\begin{split}
(\lambda_{k-1,s}-\mu_1)\|\nu\|^2_{L^2(\Omega)}+O(\epsilon^{\frac{N-2s}{2}})\|\nu\|_{L^2(\Omega)}\leq \frac{1}{4(\lambda_{k-1,s}-\mu_1)}O(\epsilon^{N-2s})=O(\epsilon^{N-2s}).
\end{split}
\end{equation*}
From Lemma \ref{lemma4.5}, we get\\
Case 1: $N>4s$,
\begin{equation*}
\begin{split}
M_\epsilon &\leq
\left(S^H_s -\mu_1 C_s\epsilon^{2s}+O(\epsilon^{N-2s})\right)\left(1+C_4\|\nu\|_{L^2(\Omega)}O(\epsilon^{\frac{N-2s}{2}})\right)+(\lambda_{k-1,s}-\mu_1)\|\nu\|^2_{L^2(\Omega)}\\
&~~~+O(\epsilon^{\frac{N-2s}{2}})\|\nu\|_{L^2(\Omega)}\\
&\leq S^H_s -\mu_1 C_s\epsilon^{2s}+O(\epsilon^{N-2s})\\
&< S^H_s,
\end{split}
\end{equation*}
for sufficiently small $\epsilon>0$  and $\mu_1>0$.\\
Case 2: $N=4s$,
\begin{equation*}
\begin{split}
M_\epsilon &\leq \left(S_s^H -\mu_1 C_s\epsilon^{2s}|log\epsilon|+O(\epsilon^{2s})\right)\left(1+C_4\|\nu\|_{L^2(\Omega)}O(\epsilon^{\frac{N-2s}{2}})\right)+(\lambda_{k-1,s}-\mu_1)\|\nu\|^2_{L^2(\Omega)}\\
&~~~+O(\epsilon^{\frac{N-2s}{2}})\|\nu\|_{L^2(\Omega)}\\
&\leq S_s^H -\mu_1 C_s\epsilon^{2s}|log\epsilon|+O(\epsilon^{2s})\\
&<S_s^H,
\end{split}
\end{equation*}
for sufficiently small $\epsilon>0$ and $\mu_1>0$.\\
Case 3: $2s<N<4s$,
\begin{equation*}
\begin{split}
M_\epsilon &\leq \left(S_s^H +\epsilon^{N-2s}(O(1)-\mu_1 C_s)+O(\epsilon^{2s})\right)\left(1+C_4\|\nu\|_{L^2(\Omega)}O(\epsilon^{\frac{N-2s}{2}})\right)+(\lambda_{k-1,s}-\mu_1)\|\nu\|^2_{L^2(\Omega)}\\
&~~~+O(\epsilon^{\frac{N-2s}{2}})\|\nu\|_{L^2(\Omega)}\\
&\leq S_s^H +\epsilon^{N-2s}(O(1)-\mu_1 C_s)+O(\epsilon^{2s})\\
&<S_s^H,
\end{split}
\end{equation*}
for sufficiently small $\epsilon>0$ and $\mu_1$ large enough.

\end{proof}

\noindent\textbf{Proof of Theorem 1.4}
By Lemma \ref{lemma5.6} and Lemma \ref{proposition5.7}, we have $J_s$ satisfies the geometric structure of the Linking Theorem,
so the Linking critical level of $J_s$, i.e.
\begin{equation*}
\begin{split}
c=\inf_{h\in\Gamma}\max_{(u,v)\in Q}J_s(h(u,v)),
\end{split}
\end{equation*}
where
\begin{equation*}
\begin{split}
\Gamma=\{h\in C(\overline{Q},Y):h=id \text{ on } \partial Q\},
\end{split}
\end{equation*}
and
\begin{equation*}
\begin{split}
Q=(\overline{B}_R\cap V)\oplus\{r(\widetilde{z}_\epsilon,0):0<r<R\}.
\end{split}
\end{equation*}
Notice that, for all $h\in\Gamma$, we have
\begin{equation*}
\begin{split}
c=\inf_{h\in\Gamma}\max_{(u,v)\in Q}J_s(h(u,v))\leq \max_{(u,v)\in Q}J_s(h(u,v)).
\end{split}
\end{equation*}
Let $\mathbb{F}_\epsilon$ as in Remark \ref{remark5.8} with $\epsilon$ sufficiently small. Since $Q\subset (\mathbb{F}_\epsilon)^2$,  taking
$h=id$ and
recalling that $(\mathbb{F}_\epsilon)^2$ is a linear subspace, we obtain
\begin{equation*}
\begin{split}
c&\leq \max_{(u,v)\in (\mathbb{F}_\epsilon)^2,(u,v)\neq(0,0)}J_s(h(u,v))=\max_{(u,v)\in (\mathbb{F}_\epsilon)^2,\eta\neq
0}J_s(|\eta|(\frac{u}{|\eta|},\frac{v}{|\eta|}))\\
&=\max_{(u,v)\in (\mathbb{F}_\epsilon)^2,\eta> 0}J_s(\eta(u,v))\leq \max_{(u,v)\in (\mathbb{F}_\epsilon)^2,\eta\geq
0}J_s(\eta(u,v)).
\end{split}
\end{equation*}
We claim that
\begin{equation*}
\begin{split}
\max_{(u,v)\in (\mathbb{F}_\epsilon)^2,\eta\geq
0}J_s(\eta(u,v))<\frac{N+2s-\mu}{2N-\mu}\left(\frac{\widetilde{S^H_\xi}}{2}\right)^{\frac{2N-\mu}{N+2s-\mu}}.
\end{split}
\end{equation*}
To verify the Claim, fixed $U=(u,v)\in (\mathbb{F}_\epsilon)^2$ such that $uv\neq 0$, by \eqref{2.5}, for all $r\geq 0$,
we infer
\begin{equation*}
\begin{split}
J_s(rU)&\leq\frac{r^2}{2}(\|U\|_Y^2-\mu_1\|U\|^2_{(L^2(\Omega))^2})-\frac{r^{2\cdot 2^*_\mu}}{
2^*_\mu}\Big(\int_{\Omega}\int_{\Omega}\frac{|u(x)|
^{2^*_\mu}|v(x)|^{2^*_\mu}}{|x-y|^\mu}dxdy\\&+\xi_1\int_{\Omega}\int_{\Omega}\frac{|u(x)|^{2^*_\mu}|u(x)|^{2^*_\mu}}{|x-y|^\mu}dxdy
+\xi_2\int_{\Omega}\int_{\Omega}\frac{|v(x)|^{2^*_\mu}|v(x)|^{2^*_\mu}}{|x-y|^\mu}dxdy\Big)\\
&:=\frac{Ar^2}{2}-\frac{r^{2\cdot 2^*_\mu}B}{ 2^*_\mu}:=g(r).
\end{split}
\end{equation*}
Notice that $r_0=\left(\frac{A}{2B}\right)^{\frac{1}{2\cdot 2^*_\mu-2}}$  is the maximum point of $g(r)$,  which
maximum value is given by
\begin{equation*}
\begin{split}
\frac{N+2s-\mu}{2N-\mu}\left(\frac{A}{2B^{\frac{1}{2^*_\mu}}}\right)^{\frac{2^*_\mu}{ 2^*_\mu-1}}.
\end{split}
\end{equation*}
Then
\begin{equation*}
\begin{split}
\max_{r\geq0}J_s(rU)\leq
\frac{N+2s-\mu}{2N-\mu}\left\{\frac{\|U\|_Y^2-\mu_1\|U\|^2_{(L^2)^2}}{2(B(u,v)+\xi_1B(u,u)+\xi_2B(v,v))^{\frac{1}{2^*_\mu}}}
\right\}^{\frac{2^*_\mu}{ 2^*_\mu-1}},
\end{split}
\end{equation*}
Therefore,  it is sufficient to show that
\begin{equation*}
\begin{split}
\widetilde{M_\epsilon}:=\max_{(u,v)\in
(\mathbb{F}_\epsilon)^2}\frac{\|U\|_Y^2-\mu_1\|U\|^2_{(L^2)^2}}{2(B(u,v)+\xi_1B(u,u)+\xi_2B(v,v))^{\frac{1}{2^*_\mu}}}<\frac{1}{2}\widetilde{S^H_\xi}.
\end{split}
\end{equation*}
Define
\begin{equation*}
\begin{split}
M_\epsilon:&=\max_{u\in
\mathbb{F}_\epsilon\backslash\{0\}}\frac{\|u\|_X^2-\mu_1\|u\|^2_{L^2}}{(\int_{\Omega}\int_{\Omega}\frac{|u(x)|^{2^*_\mu}|u(y)|^{2^*_\mu}}{|x-y|^{\mu}}dxdy)^{\frac{1}{2^*_\mu}}}\\
&=\max_{u\in \mathbb{F}_\epsilon,\int_{\Omega}\int_{\Omega}\frac{|u(x)|^{2^*_\mu}|u(y)|^{2^*_\mu}}{|x-y|^{\mu}}dxdy=1
}(\|u\|_X^2-\mu_1\|u\|^2_{L^2}).
\end{split}
\end{equation*}
Taking $s_0,t_0>0$ as in Remark \ref{remark5.1} and $u_M$ as in Lemma \ref{lemma5.9}, then  $\widetilde{M_\epsilon}$ is achieved
by function $U_M=(s_0u_M, t_0u_M)$. Therefore, from $u_M$ as in Lemma \ref{lemma5.9} and using \eqref{5.4}, we can conclude that
\begin{equation*}
\begin{split}
\widetilde{M_\epsilon}&=\frac{1}{2}\frac{(s_0^2+t_0^2)\Big(\|(u_M,u_M)\|_Y^2-\mu_1\|(u_M,u_M)\|^2_{(L^2(\Omega))^2}\Big)}{\Big(s_0^{2^*_\mu}t_0^{2^*_\mu}+\xi_1s_0^{2\cdot2^*_\mu}+\xi_2t_0^{2\cdot2^*_\mu}\Big)^{\frac{1}{2^*_\mu}}B(u_M,u_M)^{\frac{1}{2^*_\mu}}}\\
&=\frac{1}{2}mM_\epsilon<\frac{1}{2}mS^H_s=\frac{1}{2}\widetilde{S^H_\xi},
\end{split}
\end{equation*}
if one of the following conditions holds\\
a)~$N\geq4s$ and $\mu_1>0$, or\\
b)~$2s<N< 4s$ and $\mu_1$  is large enough $(\mu_1> \lambda_{k-1,s}>0)$.\\
Now, using the Linking theorem and Lemma \ref{lemma5.5}, we conclude that problem \eqref{1.1} has a nontrivial solution with critical value $c\geq\alpha$.



\begin{thebibliography}{10}



\bibitem{MRApplebaum}
D. Applebaum,
\newblock L\'{e}vy process-from probability to finance and quantum groups,
\newblock {\em Notices Amer. Math. Soc.}, 51 (2004) 1336-1347.









\bibitem{MRBrezis}
H. Br\'{e}zis and  E. Lieb,
\newblock A relation between pointwise convergence of functions and convergence of functionals,
\newblock {\em Proc. Amer. Math. Soc.}, 88 (1983), 486-490.

\bibitem{MRellipticequations}
H. Br\'{e}zis  and  L. Nirenberg,
\newblock Positive solutions of nonlinear elliptic equations involving critical Sobolev exponents,
\newblock {\em Comm. Pure Appl. Math.}, 36 (1983) 437-477.




\bibitem{MRFortunato}
A. Capozzi, D. Fortunato  and  G. Palmieri,
\newblock An existence result for nonlinear elliptic problems involving critical Sobolev exponent,
\newblock {\em Ann. Inst. H. Poincar¨¦ Anal. Non Lin¨¦aire.}, 2 (1985), 463-470.

\bibitem{MRStruwe1}
G. Cerami, D. Fortunato and M. Struwe,
\newblock Bifurcation and multiplicity results for nonlinear elliptic problems
involving critical Sobolev exponents,
\newblock {\em Ann. Inst. H. Poincar\'{e} Anal. Non Lin\'{e}aire.}, 1 (5) (1984), 341-350.

\bibitem{MRComte}
M. Comte,
\newblock Solutions of elliptic equations with critical Sobolev exponent in dimension three,
\newblock {\em Nonlinear Anal.}, 17 (5) (1991), 445-455.

\bibitem{MRCostaSilva}
D. G. Costa and E. A. Silva,
\newblock A note on problems involving critical Sobolev exponents,
\newblock {\em Differential Integral Equations.}, 8(3) (1995) 673-679.

\bibitem{MRDAvenia}
P. DAvenia, G. Siciliano and M. Squassina,
\newblock Existence results for a doubly nonlocal equation,
\newblock {\em S$\tilde{a}$ Paulo Journal of Mathematical Sciences.}, 9 (2) (2015) 311-324.



\bibitem{MRquasilinearX}
P. Dr\'{e}bek and Y. Xi Huang,
\newblock Multiplicity of positive solutions for some quasilinear elliptic equation in $\R^N$ with critical Sobolev exponent,
\newblock {\em J. Differential Equations.}, 140 (1) (1997), 106-132.



\bibitem{MRCriticalBr}
L. F. O Faria, O. H  Miyagaki and F. R Pereira,
\newblock  Critical Brezis-Nirenberg problem for nonlocal systems,
\newblock {\em Adv.  Topological Methods in Nonlinear Analysis.}, DOI 10.12775/TMNA.2017.017, 2017.2.

\bibitem{MRNonlinearAnal}
L. F. O. Faria, O. H. Miyagaki, F. R. Pereira, M. Squassina and C. Zhang,
\newblock The Brezis-Nirenberg problem for nonlocal systems,
\newblock {\em Adv. Nonlinear Anal.}, 5 (2015), 85-103.


\bibitem{MR11}
F. Gao and M. Yang,
\newblock On the Brezis-Nirenberg type critical problem for nonlinear Choquard equation,
\newblock {\em arXiv:1604.00826v4}.

\bibitem{MR88}
F. Gao and M. Yang,
\newblock On nonlocal Choquard equations with Hardy-Littlewood-Sobolev critical exponents,
\newblock {\em  J. Math. Anal. Appl.}, 448 (2) (2017), 1006-1041.

\bibitem{MRGarroni1}
A. Garroni and S. M\"{u}ller,
\newblock $\Gamma$-limit of a phase-field model of dislocations,
\newblock {\em SIAM J. Math. Anal.}, 36 (2005) 1943-1964.





\bibitem{MRDoubly}
J. Giacomoni, T. Mukherjee and K. Sreenadh,
\newblock Doubly nonlocal system with Hardy-Littlewood-Sobolev critical nonlinearity,
\newblock {\em DOI: http://arxiv.org/abs/1711.02835v1}.

\bibitem{MRnonsymmetricterm}
J. Garc\'{\i}a Azorero and I. Peral Alonso,
\newblock Multiplicity of solutions for elliptic problems with critical exponent
or with a nonsymmetric term,
\newblock {\em Trans. Amer. Math. Soc.}, 323 (2) (1991), 877-895.

\bibitem{MRLower-order}
F. Gazzola and B. Ruf,
\newblock  Lower-order perturbations of critical growth nonlinearities in semilinear elliptic equations,
\newblock {\em  Adv. Differential Equations.}, 2 (4) (1997), 555-572.

\bibitem{MRTransAmer}
N. Ghoussoub and C. Yuan,
\newblock Multiple solutions for quasi-linear PDEs involving the critical Sobolev and
Hardy exponents,
\newblock {\em Trans. Amer. Math. Soc.}, 352 (12) (2000), 5703-5743.

\bibitem{MRm-Laplacian}
J. V. Goncalves and C. O. Alves,
\newblock  Existence of positive solutions for m-Laplacian equations in $\R^N$ involving
critical Sobolev exponents,
\newblock {\em Nonlinear Anal.}, 32 (1) (1998), 53-70.

\bibitem{MRQuasilinear111}
M. Guedda and L. V\'{e}ron,
\newblock Quasilinear elliptic equations involving critical Sobolev exponents,
 \newblock {\em Nonlinear Anal.}, 13 (8) (1989), 879-902.



\bibitem{MROncriticalsystems}
Z. Guo, S. Luo and W. Zou,
\newblock On critical systems involving frcational Laplacian,
\newblock {\em J. Math.Anal. Appl.}, 446 (1) (2017), 681-706.






\bibitem{MR9}
A. Iannizzotto, S. Mosconi and M. Squassina,
\newblock $H^s$ versus $C^0$ -weighted minimizers,
\newblock {\em NoDEA Nonlinear Differential Equations Applications.}
22  (2015), 477-497.








\bibitem{MR19}
E. Lieb and  M. Loss,
\newblock Graduate Studies in Mathematics,
\newblock {\em AMS, Providence, Rhode island.}, 2001.




\bibitem{MRspectrum}
OH. Miyagaki and FR. Pereira,
\newblock Existence results for non-local elliptic systems with nonlinearities interacting with the spectrum,
\newblock {\em Advances in Differential Equations.} 23 (7/8) (2017), 555-580.








\bibitem{MRmultip}
T. Mukherjee and K. Sreenadh,
\newblock Existence and multiplicity results for Brezis-Nirenberg type fractional Choquard equation,
\newblock {\em  NoDEA Nonlinear Differential Equations Applications Nodea.} 24 (6) (2016),  63.

\bibitem{PeSqYa2}
K. Perera, M. Squassina and Y.  Yang.
\newblock Bifurcation and multiplicity results for critical fractional
  $p$-{L}aplacian problems,
\newblock {\em Math. Nachr.},  289 (2-3), 332-342, 2016.






\bibitem{MRVariational}
R. Servadei and E. Valdinoci,
\newblock Variational methods for non-local operators of elliptic type,
\newblock {\em Discrete Cont. Dyn. Syst.}, 33 (5) (2013), 2015-2137.

\bibitem{MRTheYamabe}
R. Servadei,
\newblock The Yamabe equation in a non-local setting,
\newblock {\em Adv. Nonlinear Anal.}, 2 (3) (2013), 235-270.

\bibitem{MRresonantcase}
R. Servadei,
\newblock A critical fractional Laplace equation in the resonant case,
\newblock {\em Topol. Methods Nonlinear Anal.},
43 (1) (2014), 251-267.

\bibitem{MR41}
R. Servadei and E. Valdinoci,
\newblock On the spectrum of two different fractional operators,
\newblock {\em Proc. Roy. Soc. Edinburgh Sect.} 144 (2014), 831-855.

\bibitem{MRBrezisNirenberg}
R. Servadei and E. Valdinoci,
\newblock The Brezis-Nirenberg result for the fractional laplacian,
\newblock {\em Trans. Amer. Math. Soc.}, 367 (2015) 67-102.

\bibitem{MRBrezisNirenberg2}
R. Servadei and E. Valdinoci,
\newblock A Brezis-Nirenberg result for nonlocal critical equations in low dimension,
\newblock {\em Commun. Pure Appl. Anal.}, 12 (6) (2013) 2445-2464.


\bibitem{MRDirichlet}
E.A.B. Silva and S.H.M. Soares,
\newblock  Quasilinear Dirichlet problems in $\R^N$ with critical growth,
\newblock {\em Nonlinear Anal.}, 43 (1) (2001), 1-20.

\bibitem{MREABSilva}
E.A.B. Silva and M.S. Xavier,
\newblock Multiplicity of solutions for quasilinear elliptic problems involving critical
Sobolev exponents,
\newblock {\em Ann. Inst. H. Poincar\'{e} Anal. Non Lin\'{e}aire.}, 20 (2) (2003), 341-358.

\bibitem{MRWangY}
Y. Wang  and Y. Yang,
\newblock Bifurcation results for the critical Choquard problem involving fractional $p$-Laplacian operator,
\newblock {\em Boundary Value Problems.}, 1 (2018), 132.

\bibitem{MRZHWei}
Z.H. Wei and X.M. Wu.
\newblock A multiplicity result for quasilinear elliptic equations involving critical Sobolev exponents,
\newblock {\em Nonlinear Anal.}, 18 (6) (1992), 559-567.


\bibitem{MRMinimaxTheorem}
M. Willem,
\newblock Minimax Theorem,
\newblock {\em  Progress in Nonlinear Differential Equations and their Applications.} Inc, Boston, MA, 1996.




\end{thebibliography}
\end{document}